\documentclass[11pt,a4paper,leqno]{amsart}
\usepackage{fullpage,amsmath,amsthm,amsfonts,amscd,
latexsym,amssymb,eucal,curves}
  \usepackage[all]{xy}
\usepackage{epsfig,subfigure}
\usepackage{geometry,psfrag}
\newcommand{\ZZ}{\mathbb{Z}}

\newcommand{\CC}{\mathbb{C}}

\newcommand{\PP}{\mathbb{P}}
\newcommand{\NN}{\mathbb{N}}

\newcommand{\HH}{\mathbb{H}}

\newcommand{\QQ}{\mathbb{Q}}
\newcommand{\RR}{\mathbb{R}}

\newcommand{\FF}{\mathbb{F}}

\newcommand{\End}{{\rm End}}

\newcommand{\Gal}{{\rm Gal}}

\newcommand{\ord}{{\rm ord}}

\newcommand{\cLL}{{\mathcal L}}
\newcommand{\cOO}{{\mathcal O}}
\newcommand{\cCC}{{\mathcal C}}
\newcommand{\cEE}{{\mathcal E}}

\newcommand{\cRR}{{\mathcal R}}
\newcommand{\cSS}{{\mathcal S}}
\newcommand{\cTT}{{\mathcal T}}
\newcommand{\cXX}{{\mathcal X}}
\newcommand{\cYY}{{\mathcal Y}}

\newcommand{\cMM}{{\mathcal M}}
\newcommand{\cHH}{{\mathcal H}}

\newcommand{\SL}{{\rm SL}}

\newcommand{\sms}{\smallsetminus}

\DeclareMathOperator{\Spec}{Spec}

\newcommand{\ve}{{\varepsilon}}
\newcommand{\ol}{\overline}

\newcommand{\PSL}{\mathop{\rm PSL}\nolimits}

\newcommand{\PGL}{\mathop{\rm PGL}\nolimits}

\newcommand{\Frac}{\mathop{\rm Frac}\nolimits}

\theoremstyle{definition}
\newtheorem{Defi}{Definition}[section]

\newtheorem{Conj}[Defi]{Conjecture}

\newtheorem{Question}[Defi]{Question}
\theoremstyle{remark}
\newtheorem{Rem}[Defi]{Remark}
\theoremstyle{plain}
\newtheorem{Prop}[Defi]{Proposition}
\newtheorem{Lemma}[Defi]{Lemma}
\newtheorem{Cor}[Defi]{Corollary}
\newtheorem{Thm}[Defi]{Theorem}

\newcommand{\dR}{{\rm\scriptscriptstyle dR}}

\newcommand{\C}{{\mathcal C}}
\makeatletter
\renewcommand{\subsection}{\@startsection{subsection}{2}%
        {\z@}{-3.25ex plus -1ex minus-.2ex}{-1em}{\bf}}
\makeatother
\begin{document}{\large}

\title{Differential equations associated with  nonarithmetic Fuchsian groups}

\begin{abstract}
We describe globally nilpotent differential operators of rank $2$
defined over a number field whose monodromy group is a nonarithmetic
Fuchsian group. We show that these differential operators have an
$\cSS$-integral solution.  These differential operators are naturally
associated with Teichm\"uller curves in genus $2$. They are
counterexamples to conjectures by Chudnovsky--Chudnovsky and
Dwork. We also determine the field of moduli of primitive
Teichm\"uller curves in genus $2$, and an explicit equation in some
cases.
\end{abstract}
\date{}
\subjclass[2000]{Primary 14H25; Secondary 32G15, 12H25}
\keywords{Globally nilpotent differential operators, Teichm\"uller curves, 
$\cSS$-integral solutions}
\author{Irene I.~Bouw and Martin M\"oller} \maketitle

Let $L$ be a Fuchsian differential operator of order $2$ defined over
a number field $K$. In the literature, one finds several conjectures
which connect that $L$ ``comes from geometry'', is globally nilpotent,
or admits an integral solution. Here ``coming from geometry'' could
mean, for example, that $L$ is a direct factor of the Picard--Fuchs
differential equation of a family of curves. The most famous of these
 is Grothendieck's $p$-curvature conjecture which says that
every globally nilpotent differential equation comes from geometry.
Another conjecture says that if $L$ admits an integral solution then
$L$ comes from geometry.
We refer to \S~\ref{pcurvsec} for definitions, and
to \cite{Andre}, \cite{Andre04}, \cite{Beukers} and \cite{Katz} for
partial results and precise formulations of the conjectures.

In this paper, rather
than proving a version of these conjectures, we show the existence of
integral solutions of a certain interesting class of differential
equations which come from geometry.  Here a solution $u$ is {\em
integral} if there exists a finite set of primes $\cSS$ such that the
coefficients of $u$ are in the ring ${\mathcal O}_{\cSS}\subset K$ of
$\cSS$-integral elements. 

For hypergeometric differential equations the existence of an integral
solution is well understood.  Differential equations with $4$
singularities which admit an integral solution are very rare.  Zagier
(\cite{Beukers}, \cite{Zagier}, \S~2.4) found in a huge computer
search essentially only $6$ of such differential operators with
$\QQ$-coefficients.  All of these are pullbacks of a hypergeometric
differential operators, and are associated to families of elliptic
curves.  Other known examples of $2$nd order differential equations
with an integral solution are associated to elliptic K$3$-surfaces
(\cite{BS}). This includes the differential equations that came up in
the proof of the transcendence of $\zeta(2)$ and $\zeta(3)$.  The
differential equations with integral solution we consider in this
paper have $5$ singularities, and are not the pullback of a
hypergeometric differential equation.

Our differential equations also come from geometry, though they are of
 a different nature. Namely, they are the uniformizing differential
 equations of Teichm\"uller curves in genus $2$.  These curves were
 discovered by Calta and by McMullen and intensely studied from a
 complex-analytic point of view. This paper starts to explore
 arithmetic aspects of Teichm\"uller curves.  For an introduction to
 Teichm\"uller curves, we refer to \S~\ref{Teichsec}.  We restrict to
 Teich\-m\"uller curves parameterizing curves of genus $2$. Let $C$ be
 such a Teichm\"uller curve.

This paper contains results in two direction. Firstly, we prove new
results on Teichm\"uller curves in genus $2$. Secondly, we show that
the uniformizing differential equation of these Teichm\"uller curves
have interesting arithmetic properties.

\bigskip\noindent A new ingredient we use for studying Teichm\"uller
curves in genus $2$ is the construction of genus-$2$ fibrations as
double coverings of ruled surfaces (following \cite{CaPi07}). This
allows us, for example, to compute the Lyapunov exponents of $C$
(\S~\ref{lyasec}) for all Teichm\"uller curves $C$ in genus two.  This
extends a result of Bainbridge (\cite{Ba06}).

Teichm\"uller curves in genus $2$ whose generating translation surface
$(X,\omega)$ has a double zero are classified by McMullen
(\cite{McM03}, \cite{McM06a}). These Teichm\"uller curves are
characterized by two invariants: the discriminant $D\in \NN$ and, if
$D\equiv 1\pmod{8}$, the spin invariant $\varepsilon\in \ZZ/2\ZZ$
which is the signature of a quadratic form on a certain subspace of
$H^1(X,\ZZ/2\ZZ)$ (\S~~\ref{Teichsec}). We denote by $W_D^\varepsilon$
the Teichm\"uller curve with discriminant $D$ and spin invariant
$\varepsilon$.  Surprisingly, the field
of moduli of $W_D^\varepsilon$ depends on whether $D$ is a square or
not, even though the spin invariant may be defined the same way in both
cases.

\bigskip\noindent{\bf Theorem \ref{Galoisthm}}. {\em If $D\equiv
1\pmod{8}$ is not a square, the field of moduli of the Teichm\"uller
curves $W_D^0$ and $W^1_D$ is $\QQ(\sqrt{D})$. Otherwise, the field of
moduli of $W^\varepsilon_D$ is $\QQ$}.

\bigskip
Theorem \ref{Galoisthm} allows to significantly simplify Bainbridge's
calculation of the orbifold Euler  characteristic of
$W_D^\varepsilon$ (\cite{Ba06}). It is still an open problem
to determine what the $W_D^\varepsilon$ are, for example, as algebraic
curves marked by their cusps and elliptic points. We solve this
problem for $D=13$ and $D=17$ which are the first nontrivial cases.
Note that the usual definitions of the Teichm\"uller curves, as
Teichm\"uller geodesics and as locus of eigenforms with a double zero,
are rather unsuitable for such a computation.

To determine $W_D^\varepsilon$, we explicitly compute an equation for
the universal family $\cXX$ of curves of genus $2$ parameterized by
$W_D^\varepsilon$ (\S~\ref{eqnsec}). We show that, for $D=13, 17$,
this family of curves is determined by its degenerations. The explicit
equation for $\cXX$ allows us also to compute the uniformizing
differential equations $L_{D}^\varepsilon$ of $W_D^\varepsilon$
(\S~\ref{desec}).

In \S~\ref{integralsec} we show that $L_D^\varepsilon$ has an integral
solution in a neighborhood of a cusp. We deduce this from the
existence of an integral solution modulo $p^n$ for all $n$. The
existence of approximate solutions follows from Katz' theorem on
expansion coefficients (\cite{Katz84}). The proof of the integrality
does not use the explicit equation for the differential equation
$L_D^\varepsilon$, but only the existence of an integral model for
$\cXX$. Therefore the proof applies to all Teichm\"uller curves of
genus zero, under a mild hypotheses. (See \S~\ref{integralsec} for the
precise statement.)

\bigskip\noindent{\bf Theorem \ref{integralthm}.} {\em Let
 $W_D^\varepsilon$ be a Teichm\"uller curve of genus zero. There
 exists a finite set $\cSS\subset \ZZ$ of primes such that $L_D^\ve$
 admits a holomorphic $\cSS$-integral solution.  }

\bigskip
The differential operators $L_D^\varepsilon$ are counterexamples to
several conjectures which fit into the circle of ideas around
Grothendieck's $p$-curvature conjecture.  Recall that a differential
operator $L$ is {\em globally nilpotent} if the reduction of $L$
(modulo $p$) has nilpotent $p$-curvature for almost all $p$
(\S~\ref{pcurvsec}). The nilpotence of the $p$-curvature may also be
characterized in terms of the existence of algebraic solutions of the
reduction of $L$ modulo $p$.  The conjecture of D.V.\ Chudnovsky and
G.V.\ Chudnovsky (Conjecture \ref{Chudconj}) may be seen as an attempt
to characterize the monodromy groups of globally nilpotent
differential equations. A  theorem of Katz (\cite{Katz}) implies
that $L_D^\varepsilon$ is globally nilpotent. The following result on
the monodromy group of $L_D^\varepsilon$ yields 
counterexamples to this conjecture.

\begin{samepage}
\bigskip\noindent{\bf Theorem \ref{Chudprop}.} {\em There exist
infinitely many $D$ such that:
\begin{itemize}
\item[(a)] $\Gamma_D^\varepsilon$ is nonarithmetic,
\item[(b)] $\Gamma_D^\varepsilon$ is not commensurable to a triangle
group,
\item[(c)] the uniformizing differential equation $L_D^\varepsilon$ of
$\Gamma_D^\varepsilon$ may be defined over $\QQ(\sqrt{D})$, and is
globally nilpotent.
\end{itemize}}
\end{samepage}

We note that (a) and (b)  of Theorem
\ref{Chudprop} are due to McMullen (\cite{McM03}). Our contribution
to this theorem is linking the Teichm\"uller curves to the theory of
differential equations. 
There is a finite list of discriminants $D$ such that
$L_D^\varepsilon$ lives on a curve of genus zero. In this case the
differential equation $L_D^\varepsilon$ does not have an algebraic solution
(Proposition \ref{Dworkprop}). This yields nonarithmetic
counterexamples to a conjecture of Dwork (Conjecture \ref{Dworkconj}).

An arithmetic counterexample to Dwork's conjecture has previously been
found by Krammer (\cite{Krammer}). Krammer's example is the
uniformizing differential equation of a Shimura curve
$C=\HH/\Gamma$. Krammer's strategy in computing the uniformizing
differential equation of $\Gamma$ is to exploit the existence of a
correspondence $C'\rightrightarrows C$. A similar strategy has been
used by Elkies (\cite{Elkies}).  Krammer's strategy is unlikely to
work for a nonarithmetic Fuchsian group (Remark \ref{Daanrem}).  The
reason why it is hard to find counterexamples to the conjecture of
Chudnovsky--Chudnovsky appears to be that it is difficult to find
globally nilpotent differential equations whose monodromy group is a
nonarithmetic Fuchsian group which is not commensurable to a triangle
group.

\section{Teichm\"uller curves in genus $2$}\label{Teichsec}
A {\em Teichm\"uller curve} is a generically injective, holomorphic
map $C \to \cMM_g$ which is geodesic for the Teichm\"uller
metric. Here $C$ is a smooth algebraic curve $C$ and $\cMM_g$ is the
moduli space of curves of genus $g$. Teichm\"uller curves arise
naturally from the study of dynamics of billiard tables.  We recall
the standard construction of Teichm\"uller curves. Let $X$ be a
Riemann surface of genus $g$ and $\omega$ a holomorphic $1$-form on
$X$. The orbit of $(X, \omega)$ under the natural action of
$\SL_2(\RR)$ on $\Omega {\mathcal T}_g$ projects to a geodesic disc
$\HH \to \cTT_g \to \cMM_g$ under the Teichm\"uller metric. If the
stabilizer $\Gamma$ of $(X, \omega)$ is a lattice in $\SL_2(\RR)$, the
image of the orbit in $\cMM_g$ is a curve. Its normalization
$C=\HH/\Gamma$ is then a Teichm\"uller curve.  A pair $(X, \omega)$ is
called a {\em translation surface}. A translation surface that
generates a Teichm\"uller curve is called a {\em Veech surface}. The
corresponding lattice $\Gamma$ is called the {\em affine group} of
$C$.

More generally, one could consider Teichm\"uller curves generated by
 $(X, q)$, where $q\in \Gamma(X, (\Omega_X^1)^{\otimes 2})$ is a
 quadratic differential form. After replacing $C$ by a cover of
 degree $2$ if necessary, one may assume that there exists a $1$-form
 $\omega$ with $q=\omega^2$. Therefore it is no restriction to only consider 
  Teichm\"uller curves $C$ which are generated by a translation surface 
$(X,\omega)$.

We let $\ol{C}$ be a smooth compactification of $C$ and $S := \ol{C}
\sms C$. We may replace $C$ by a finite, unramified cover such that
there exists a universal family $f:\cXX\to C$ of smooth curves of
genus $g$. Moreover, we may assume that $C = \HH/\Gamma$ with $\Gamma$
torsion free. Then $f:\cXX\to C$ extends to a family $\ol{f}:
\ol{\cXX} \to \ol{C}$ of stable curves (\cite{Mo06a}, \S~1.4).

By \cite{Mo06a}, Theorem 2.6, Teichm\"uller curves
parameterize curves whose Jacobian has a factor of rank $r$ with real
multiplication. Here $r \leq g$ is the degree of the trace field of
$\Gamma$. (Recall that the {\em trace field} is the (finite) extension
of $\QQ$ which is generated by the traces of all $A\in \Gamma$.)  If
$(X, \omega)$ is a translation surface generating a Teichm\"uller
curve, then $\omega$ is an eigenform for the real multiplication.

We suppose from now on that $g=2$ which is the case  we
study in this paper.  The  de Rham cohomology of $f:\ol{\cXX}\to \ol{C}$
decomposes as
\begin{equation}\label{decompeq}
\cHH^1_\dR(\ol{\cXX})=\cEE_1\oplus\cEE_2,
\end{equation}
where $\cEE_i$ are flat vector bundles of rank $2$ with logarithmic
poles in $S$.  We denote their Hodge filtration by $\cLL_i\subset
\cEE_i$.  In particular, we have
\begin{equation}\label{Hodgeeq}
f_\ast \omega_{\ol{\cXX}/\ol{C}}=\cLL_1\oplus\cLL_2.
\end{equation}

It is shown in \cite{Mo06a}, Theorem 2.6, that one of the vector
bundles $\cEE_i$, say $\cEE_1$, is {\em indigenous}. This means that
the {\em Kodaira--Spencer map}
\begin{equation} \label{KSeq}
\Theta: \cLL_1 \to \cEE_1 \stackrel{\nabla}{\to}
 \cEE_1 \otimes \Omega^1_{\ol{C}}(\log S) 
\to (\cEE_1/\cLL_1) \otimes \Omega^1_{\ol{C}}(\log S)
\end{equation}
is an isomorphism. (The local system corresponding to $\cEE_1$ is
called {\em maximal Higgs}.)  One may characterize Teichm\"uller curves 
generated by translation surfaces 
via the existence of an indigenous bundle $\cEE_1$ (\cite{Mo06a}, Theorem 5.5).

\begin{Rem} \label{indicritrem}  After replacing $C$ by a finite
unramified cover, there exist
isomorphisms $\cLL_1\simeq \Omega^1_{\ol{C}}(\log S)^{1/2}$ and
$\cEE_1/\cLL_1\simeq \Omega^1_{\ol{C}}(\log S)^{-1/2}$. In particular,
$\deg{\Omega^1_{\ol{C}}(\log S)}=2g(\ol{C})-2+|S|$ is even.
\end{Rem}

The rational numbers $\lambda_1:=1$ and
$\lambda_2:=\deg(\cLL_2)/\deg(\cLL_1)$ are called the {\em Lyapunov
exponents} of $C$. We will use the Lyapunov exponents
rather then the degrees of the line bundles $\cLL_i$  as invariants as the
Lyapunov exponents do not change if we replace $\ol{C}$ by a finite
cover.  We refer to \cite{BoMo05}, Proposition 8.5, for the proof that
these numbers coincides with the usual definition of the Lyapunov
exponents as growth rates of the Hodge norms along the Teichm\"uller 
geodesic flow.

We suppose that $(X, \omega)$ generates the Teichm\"uller curve $C$,
via the above construction. Then $X$ is a fiber of $\cXX$ and $\omega$
is a section of $\cLL_1$. There are two possibilities: $\omega$ has
either two simple zeros or one double zeros.  We denote by $\Omega
\cMM_2(1,1)$ (resp.\ $\Omega\cMM_2(2)$) the locus of pairs $(X,
\omega)$ where $\omega$ has two simple  zeros (resp.\ one double zero).

\bigskip\noindent
We recall McMullen's classification of primitive Teichm\"uller curves
in the moduli space of curves of genus $g=2$.  A
Teichm\"uller curve is {\em primitive} if it does not arise from a
family of curves of lower genus via a branched covering. 

There is a unique primitive Teichm\"uller curve corresponding to a
translation surface $(X, \omega)$ for which $\omega$ has two simple
zeros (\cite{McM06a}). In this case, the genus of $C$ is zero and the
connection on $\cEE_1$ has three regular singularities. Therefore the
affine group of this Teichm\"uller curve is a triangle group. This
family is a special case of the families studied in \cite{BoMo05}. The
curves parameterized by $C$ have real multiplication by
$\ZZ[\sqrt{5}]$.

There is an infinite family of primitive Teichm\"uller curves
corresponding to a translation surface $(X, \omega)$ for which
$\omega$ has a double zero.  McMullen (\cite{McM05a}) shows that these
Teichm\"uller curves are characterized by two invariants: the
discriminant $D$ and the spin invariant $\epsilon$. We recall the
definition of these invariants.

 Let $K$ be a totally real number field with $[K:\QQ]=2$, and let
$\cOO_D\subset K$ be an order of discriminant $D$.  We denote by
$W_D\subset \cMM_g$ the locus of curves of genus $2$ whose Jacobian
admits real multiplication by $\cOO_D$, and which carry an eigenform
$\omega$ for the real multiplication which has a double zero at one of
the $6$ Weierstra\ss\ points. Then $W_D=\emptyset$ if $D\leq 4$. For
$D\geq 5$, every irreducible component of $W_D$ is a Teichm\"uller
curve. If $D\equiv 1\pmod{8}$ and $D\neq 9$, then $W_D$ is the
disjoint union of two curves $W_D^\varepsilon$, where $\varepsilon\in
\ZZ/2\ZZ$ is the spin invariant. Otherwise, $W_D$ is irreducible. The
spin invariant $\ve$ may also be defined more conceptually as the
Arf invariant of some real multiplication endomorphism acting on
$\ZZ/2\ZZ$-cohomology (\cite{McM05a}, \S~5).

To avoid a case distinction, we denote all primitive Teichm\"uller
curves by $W_D^\varepsilon$ even if $W_D$ is irreducible. The curve
$W_D^\varepsilon$ may be defined over a number field. (See Theorem
\ref{Galoisthm} for a precise statement.)

\section{Computation of the Lyapunov exponents}\label{lyasec}
We let $f:\ol{\cXX}\to \ol{C}$ be the universal family over a (finite cover of
a) Teichm\"uller curve, as in \S~\ref{Teichsec}.  Let $(X, \omega)$ be
a translation surface generating this Teichm\"uller curve.  In this
section, we compute the Lyapunov exponents $\lambda_i$ of $C$ by using
a result of Catanese and Pignatelli (\cite{CaPi07}) on the structure
of genus-$2$ fibrations. This gives a new, shorter proof of a result
of Bainbridge (\cite{Ba06}) in the case that $(X,
\omega)\in\Omega\cMM_2(2)$. In the case that $(X,
\omega)\in\Omega\cMM_2(1,1)$ and $C$ is imprimitive the result is
new. The affine group of
the (unique) primitive Teichm\"uller curve 
corresponding to a translation surface in $\Omega\cMM_2(1,1)$ is a
triangle group. Its Lyapunov exponents were calculated in
\cite{BoMo05}.

\begin{Lemma}\label{singfiberslem}
Let $c\in \ol{C}$ be a point such that the fiber  $X=\ol{\cXX}_c$ is singular.
\begin{itemize}
\item[(a)]  The curve $X$ does not
contain a separating node. In particular, $X$ does not consist of
two elliptic curves meeting in one point.
\item[(b)] If $\ol{C}$ is a primitive Teichm\"uller curve 
generated by a Veech surface in $\Omega\cMM_2(2)$
then $X$
consists of a projective line which intersects itself in two points.
\end{itemize}
 \end{Lemma}

\begin{proof}
This lemma follows from some well-known results on the geometry of
translation surfaces $(X,\omega)$ which we quickly recall.  For a
given direction in $v \in \RR^2$ one considers the geodesics for
$\omega$ with slope $v$. A direction is called {\em periodic}, if all
geodesic are closed or join zeros of $\omega$. The translation surface
decomposes into cylinders in the direction of a periodic direction,
as, for example, in the horizontal direction in
Figure~\ref{cap:spl_prototype}. A geodesic in the interior of a
cylinder is called a {\em core curve}.  

 The singular fibers of $\ol{f}:\ol{\cXX}\to \ol{C}$ correspond
bijectively to periodic directions, up to the action of the affine
group $\Gamma$.  Topologically, singular fibers of
$\ol{f}:\ol{\cXX}\to \ol{C}$ are obtained by squeezing the core curves
of the cylinder decomposition of the corresponding flat surface in
some direction. This is justified in \cite{Ma75}, where flat metrics
are related to hyperbolic metrics.

 Suppose that a
core curve of a cylinder $Z$ separates $(X, \omega)$ into two planar
polygons $P_1$ and $P_2$. This implies that all sides of $P_1$, except
for the boundary of $Z$, come in pairs that are glued by translations
preserving the global orientation of the plane.  Fix an orientation of
the boundary of $P_1$. Sides glued together consequently have opposite
orientation on the boundary. We have obtained a contradiction since
the translation vectors of all sides of $P_1$ have to add up to zero.

The second statement follows from \cite{Mo07}, Theorem~2.1 and
Corollary~2.2.
\end{proof}

The relative canonical map of a family $\ol{f}: \ol{\cXX} \to \ol{C}$
of genus-two curves defines a rational map $\varphi:\ol{\cXX} - - >
\PP:=\PP(V_1)$, where $V_1 = f_* \omega_{\cXX/C}$.  It is known since
the work of Horikawa how to reconstruct $\ol{f}$ from the ruled
surface $\PP$ and covering data. We follow the recent nice account in
\cite{CaPi07}.

 By Lemma~\ref{singfiberslem} the cokernel
$$\tau= {\rm coker}(S^2( f_* \omega_{\cXX/C}) \to f_*
\omega^2_{\cXX/C} ) $$ is zero and for the same reason $\varphi$ is
actually a morphism.  In this situation, the relative canonical ring
of the fibration is uniquely determined (\cite{CaPi07}, page 1014 and
Proposition 4.8) by a morphism
$$\delta: \det(V_1)^2 \to S^6(V_1).$$
Since in our situation $V_1 = \cLL_1 \oplus \cLL_2$ splits into
eigenspaces of real multiplication, $\delta$ is a direct sum
of maps between line bundles. The structure of $\delta$
is described in the following Propositions~\ref{deltaprop} and
\ref{deltaprop2}.

\begin{Prop}\label{deltaprop}
Let $\cXX \to C$ be the universal family over a Teichm\"uller curve
generated by a translation surface in $\Omega \cMM_2(2)$.  Then we
have a decomposition
$$ \delta=\oplus_{k=0}^6 \delta(k): \cLL_1^2 \otimes \cLL_2^2
\to \oplus_k (\cLL_1^{6-k} \otimes \cLL_2^{k}) $$
with the properties
\begin{itemize}
\item[(a)] the map $\delta(0)$ is identically zero, and
\item[(b)] the map $\delta(1)$ is an isomorphism.
\end{itemize}
\end{Prop}

\begin{proof}
The decomposition of $\delta$ follows immediately from
the decomposition of $f_* \omega_{\ol{\cXX}/\ol{C}}$.

Fix a point $p \in \ol{C}$ and let $t$ a local parameter
at $p$. Choose local sections $s_i$ of $\cLL_i$ in
a neighborhood of $p$, and write
$$\delta(k)(s_1^{\otimes 2}s_2^{\otimes 2}) = c_{k}(t)
s_1^{\otimes 6-k}s_2^{\otimes(k)}$$ for functions $c_k$,
$(k=0,\ldots,6)$. The choice of local coordinates $c_k$ is equivalent to
representing $\ol{f}$ in a neighborhood of $p$ in terms of the inhomogeneous
coordinate $x=s_1/s_2$ as 
\begin{equation} \label{HENF}
y^2 = \sum_{k=0}^6 c_{k} x^k.
\end{equation}
Compare to \cite{CaPi07}, \S~4.

Recall from \S~\ref{Teichsec} that there exists a section $\omega_1$ of
$\cLL_1$ which has a double zero. The choice of the coordinate $x$
implies that $\omega_1$ has a double zero at $x=\infty$, hence 
$c_6=0$. This implies (a). Since (\ref{HENF}) represents a family of
curves of genus $2$, we conclude that $c_5\neq 0$. This implies (b).
\end{proof}

A similar proof yields the analogous result in the case of
Teichm\"uller curves generated by a translation surface in
$\Omega\cMM_2(1,1)$.

\begin{Prop}\label{deltaprop2}
Let $\cXX \to C$ be the universal family over a Teichm\"uller
curve generated by a translation surface in $\Omega M_2(1,1)$. 
Then we have a decomposition
$$ \delta=\oplus_{k=0}^6 \delta(k): \cLL_1^2 \otimes \cLL_2^2 \to
\oplus_k (\cLL_1^{6-k} \otimes \cLL_2^{k}),$$ and $\delta(0)$ is an
isomorphism.
\end{Prop}

We now use Propositions~\ref{deltaprop} and \ref{deltaprop2} to 
calculate the Lyapunov exponents for
Teich\-m\"uller curves in genus two. 

\begin{Cor}\label{lyacor}
Let $C$ be a Teichm\"uller curve in genus $g=2$ generated by the
translation surface $(X,\omega)$.  The Lyapunov exponents are
\[
(\lambda_1, \lambda_2)= \begin{cases}\{1,1/3\}&\text{ if }\, (X,\omega)
\in \Omega M_2(2),\\ 
\{1, 1/2\}& \text{ if }\, (X,\omega) \in \Omega
M_2(1,1).
\end{cases}
\]
\end{Cor}

\begin{proof}
By definition, we have that $\lambda_1=1$ and
$\lambda_2=\deg(\cLL_2)/\deg(\cLL_1)$. In the situation of Proposition
\ref{deltaprop} we have that
\[
2\deg(\cLL_1) +2\deg(\cLL_2) = \deg(\cLL_1) + 5 \deg(\cLL_2).
\]
In the situation of Proposition \ref{deltaprop2}, we have that
\[
2\deg(\cLL_1)+2\deg(\cLL_2)=6\deg{\cLL_2}.
\]
This implies the statement.
\end{proof}

\section{Prototypes for singular
fibers and the Galois action on the set of components of
$W_D$}\label{protosec} In this section, we consider primitive
Teichm\"uller curves corresponding to a translation surface $(X,
\omega)\in \Omega\cMM_2(2)$.  Let $W_D$ be as in \S
\ref{Teichsec}. Recall that $W_D$ is the disjoint union of two
Teichm\"uller curves $W_D^\varepsilon$ if $D\equiv 1\pmod{8}$ and is a
Teichm\"uller curve otherwise. The goal of this section is to show
that if $D\equiv 1\pmod{8}$ is not a square the field of moduli
$W_D^\varepsilon\to \cMM_2$ is not $\QQ$. A key
ingredient is a normal form for the degenerate fibers of the
corresponding universal family $\ol{f}:\ol{\cXX}^\varepsilon(D)\to
\ol{C}(D)$. We usually drop $D$ and $\varepsilon$ from the notation, if
they are clear from the context.

McMullen (\cite{McM05a}) shows that the cusps of $W_D$ correspond to
so-called splitting prototypes. We recall the definition.

\begin{Defi}\label{protodef}
A quadruple of integers $(a,b,c,e)$ is a {\em splitting prototype of
discriminant} $D$ if
\begin{align*}
&D=e^2+4bc, \qquad 0 \leq a < \gcd(b,c),\qquad 0<b,\\
&0< c, \qquad c+e < b \quad \text{ and } \quad\gcd(a,b,c,e)=1.\\
\end{align*}
\end{Defi}

To a splitting prototype $(a,b,c,e)$ of discriminant $D$, we may
associate a translation surface as in Figure~\ref{cap:spl_prototype},
where vertical sides are glued by horizontal translations and where
$\lambda = (e +\sqrt{D})/2$.
\begin{figure}[here]
\label{fig:spl_prototype}
\begin{center}
\psfrag{a}{$\lambda$}
\psfrag{b}{$\lambda$}
\psfrag{c}{$(a,c)$}
\psfrag{d}{$(b,0)$}
\epsfig{figure=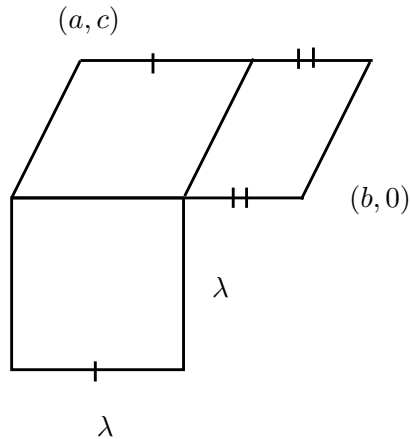,width=5cm}
\end{center}
\caption{A translation surface attached to a splitting prototype}
\label{cap:spl_prototype}
\end{figure}

Let $(a,b,c,e)$ be a splitting prototype of discriminant $D$. Write
$D=EF^2$ with $E$ square free. Then it follows from \cite{McM05a},
Proposition 5.3, that the corresponding translation surface has
spin invariant
\begin{equation} \label{spinformulaeq}
 \ve = \frac{e-F}{2} + (c+1)(a+b+ab) \pmod{2}.
\end{equation}

\begin{Prop}\label{singnormalprop}
\begin{itemize}
\item[(a)]
The singular fiber corresponding to  $(a,b,c,e)$ 
may be given by
\begin{equation} \label{normalformeq}
y^2 = (x-\mu)(x-\mu-b^2)^2(x-\mu-\lambda^2)^2 \quad
\text{where} \quad
\mu = \frac{\lambda (\lambda\sigma(\lambda)-b^2)}{\sigma(\lambda)
-\lambda}.
\end{equation}
Here $\sigma$ denotes  the generator of $\Gal(\QQ(\sqrt{D}), \QQ)$.
\item[(b)] We may, furthermore, suppose that $\omega_1={\rm d}x/y$ and
$\omega_2=x\,{\rm d}x/y$ are eigenforms for real multiplication by
${\mathcal O}_D$.
\item[(c)] The variable $x$ is uniquely determined by (a) and (b), up
to multiplication by a constant.
\end{itemize}
\end{Prop}

\begin{proof}
Lemma \ref{singfiberslem}.(b) implies that a singular fiber $X$ of
$\ol{\cXX}$ is a projective line $\PP^1_z$ with two pairs of points
identified.  We normalize the hyperelliptic involution ad hoc by $z
\mapsto -z$. The Weierstra\ss\ points on the generic fiber of $\ol{\cXX}$
specialize to the points $0, \infty$ and the two pairs of points $p,-p$
and $q,-q$ that are glued together.

Let $\omega_1$ be a section of $\cLL_1$ restricted to the singular
fiber $\cXX_c$. We claim that we can read off the residues of
$\omega_1$ from the geometry of Figure~\ref{cap:spl_prototype}, up to
a common scalar multiple coming from replacing $\omega_1$ by a scalar
multiple and up to $\pm 1$ stemming from the choice of a branch at the
singularity. Recall that the singular fiber is obtained topologically
by squeezing the horizontal core curves of the cylinders. Metrically,
${\rm diag}(e^t, e^{-t})$ tends  to $\cXX_c$ for $t \to \infty$
(\cite{Ma75}). Consequently, up to indeterminacies mentioned in the
claim, the residues of $\omega_1$ are the lengths of the core curves
of the horizontal cylinders in Figure~\ref{cap:spl_prototype}.  We
choose one of the two possibilities for the sign, and find that
$$\omega_1 = \left(\frac{\lambda}{z-p}+ \frac{-\lambda}{z+p} +
\frac{-b}{z-q} + \frac{b}{z+q} \right){\rm d}z. $$ 

By assumption, $\omega_1$ has a double zero, which is either at $x=0$
or at $x=\infty$. We may assume that $\omega_1$ has a double zero at
$x=\infty$. This implies that $\lambda p=b q$. After multiplying $z$
by a suitable constant, we may assume that $q=\lambda$ and $p=b$.

Since $b\in \QQ$, it follows that $b$ is fixed by $\sigma$. Therefore
$$\omega_2 = \left(\frac{\sigma(\lambda)}{z-p}+
\frac{\sigma(-\lambda)}{z+p} + \frac{-b}{z-q} +
\frac{b}{z+q}  \right){\rm d}z. $$
Note that $\omega_2$ has  zeros  in  $z=\pm \sqrt{\mu}$, for some $\mu$.

In terms of the coordinates $x'=z^2$ and $y=z(x'-\lambda^2)(x'-b^2)$,
we find that, up to multiplicative constants, $\omega_1={\rm d}x/y$
and $\omega_2 = (x'-\mu)\, {\rm d}x'/y$. Therefore defining $x=x'+\mu$ yields
(b). The coordinates $x$ and $y$ obviously satisfy the relation in (a).
 The uniqueness assertion (c) follows immediately from our normalization.
\end{proof}

\begin{Thm}\label{Galoisthm}
\begin{itemize}
\item[(a)] Suppose that $D=d^2 \equiv 1 \pmod{8}$. Then the components
 $W_D^\ve$ of $W_D$ are fixed by the absolute Galois group
 $G_\QQ$, i.e.\ $G_\QQ$ acts trivially on the map $W_D^\ve \to \cMM_2$.
 \item[(b)] If $D \equiv 1 \mod 8$ is not a square, the action of
 $\Gal(\QQ(\sqrt{D})/\QQ)$ sends $W_D^0\to\cMM_2$ to
 $W_D^1\to\cMM_2$. In particular, the field of moduli of  
$W_D^\ve \to \cMM_2$ is not $\QQ$.
\end{itemize}
\end{Thm}

This theorem may be reformulated as follows. The moduli map
$W_D^\varepsilon\to\cMM_2$ may be defined over $\QQ$ if and only if
$W_D^\ve$ is imprimitive or $W_D$ is irreducible. An immediate
consequence is that orbifold the Euler characteristics of
$W_D^0$ and $W_D^1$ coincide. The proof in \cite{Ba06} 
needed a careful analysis of intersection numbers to arrive at
this conclusion.

\begin{proof}
Suppose that $D\equiv 1\pmod{8}$ is a square. Then $W_D^\varepsilon$
is imprimitive, and the theorem follows from (\cite{McM05a}, Theorem
6.1).

Suppose that $D\equiv 1\pmod{8}$ is not a square.  For the second
statement, recall from \cite{Mo06a}, Corollary~5.4, that $G_\QQ$ acts on
the set of Teichm\"uller curves. Moreover, given a family of curves
over a Teichm\"uller curve, $G_\QQ$ acts on the family of Jacobians, and
maps eigenforms for real multiplication to eigenforms. Obviously,
$G_\QQ$ also preserves the multiplicity of the zeros of both eigenforms.

By the classification of Teichm\"uller curves in $\Omega \cMM_2(2)$, the
absolute Galois group $G_\QQ$ acts either trivially or via a group of
order $2$. Let $c\in \ol{C}$ be a cusp. Then to $c$ we associate a
triple $(X, \omega_1, \omega_2)$, where $X$ is the stable curve
$\ol{\cXX}_c$ and $\omega_i$ are holomorphic sections of $\cEE_i|_{c}$
which are eigenforms for real multiplication.

 The boundary divisors of $\ol{\cMM}_2$ are defined over $\QQ$. To
show that the $G_\QQ$-action is non-trivial, it suffices therefore to
find a cusp $c_0$ of $W_D^0$ whose Galois conjugate is a cusp $c_1$ of
$W_D^1$ and such that the triples $(X_j, \omega_{1, j}, \omega_{2,
j})$ are not isomorphic.

We consider the splitting prototypes $(0,(D-1)/4,1,1)$ and
$(0,(D-1)/4,1,-1)$ which have different spin invariant by
(\ref{spinformulaeq}).  
We write $\lambda_e=(e+\sqrt{D})/2$ for $e=\pm 1$. 
Let $\mu_e$ be given as in equation \eqref{normalformeq}.
One easily calculates that
$$ \sigma(\lambda_1^2) = \lambda_{-1}^2 \quad \text{and} \quad
\sigma(\mu_1) = \mu_{-1}.$$
Consequently, $\sigma$ applied to the equation \eqref{normalformeq}
for the prototype $(0,(D-1)/4,1,1)$ yields the corresponding
equation attached to the prototype $(0,(D-1)/4,1,-1)$. We thus
have found a cusp that is moved by $G_\QQ$ from $W_D^0$ 
to $W_D^1$.

It remains to check that the triples attached to the prototypes
are not isomorphic. Equivalently, we have to check that the curves 
$X^{\varepsilon(e)}$ given by the normal form \eqref{normalformeq}
are not isomorphic through an isomorphism fixing $0$ and $\infty$.
Suppose not. Then the
 isomorphism $X^{\varepsilon(-1)}\to X^{\varepsilon(1)}$
 corresponds to the change of coordinates $x\mapsto
 x\mu_1/\mu_{-1}$. Since $\mu_1=\sigma(\mu_{-1})\neq \mu_{-1}$, we
 find that
\[
 \frac{\mu_1}{\mu_{-1}}(\mu_{-1}+b^2) = (\mu_{1}+\lambda_1^2).
\]
This is equivalent to
\[
\frac{\mu_1}{\mu_{-1}}=\frac{\lambda_1^2}{b^2}.
\]
One computes that
\[
\frac{\mu_1}{\mu_{-1}}=-\frac{\lambda_1}{\lambda_{-1}}=
\frac{(1+\sqrt{D})^2}{D-1},
\quad\text{ and }\quad \frac{\lambda_1^2}{b^2}=\frac{4(1+\sqrt{D})^2}{D-1}.
\]
We obtain a contradiction.
\end{proof}

\section{Triangle groups}\label{trianglesec} 
In this section we recall a result of McMullen which says that there
are infinitely many Teichm\"uller curves in genus $2$ whose affine
group is not commensurable to a triangle group. This is a key
ingredient in showing that infinitely many Teichm\"uller curves in
genus $2$ are counterexamples to the conjectures of Dwork and
Chudnovsky--Chudnovsky  (\S~\ref{counterexasec}).

Let $f:\cXX\to C$ be the universal family over (a finite cover of) a
primitive Teichm\"uller curve in genus $g=2$, as in \S
\ref{Teichsec}. The {\em elliptic
points} are the images in $C$ of the points in the complex upper half
plane $\HH$ with nontrivial stabilizer. If $z\in \HH$ has nontrivial
stabilizer $\Gamma_z$, then $\Gamma_z$ is finite.  After replacing $C$
by a finite cover, one may suppose that $\Gamma$ is torsion free. If
$\Gamma$ is torsion free there are no elliptic points. This will
sometimes be convenient. Recall that $K=\QQ(\sqrt{D})$ is the
trace field of $\Gamma(W_D^\ve)$.

A Fuchsian group $\Gamma$ is a {\em triangle group} if $g(\ol{C})=0$ and the
set $\{\tau_i\}$ of cusps and elliptic points has cardinality $3$. A
triangle group is denoted by $\Gamma=\Delta(n_1,n_2,n_3)$, where
$n_i=\infty$ if $\tau_i$ is a cusp and $n_i$ is the order of the
stabilizer if $\tau_i$ is an elliptic point.  Teichm\"uller curves
always have cusps, therefore we may assume that $n_3=\infty$
(\cite{Veech}). The trace field of the triangle group
$\Delta(n_1,n_2,\infty)$ is $\QQ(\zeta_{n_1}+\zeta_{n_1}^{-1},
\zeta_{n_2}+\zeta_{n_2}^{-1})$, where $\zeta_{n_j}$ is a primitive
$n_j$th root if unity of $n_j$ is finite and is  $1$ otherwise.

The following lemma is proved by McMullen (\cite{McM03}, \S~9). We
recall the proof for the convenience of the reader.

\begin{Lemma}\label{trianglelem}
Let $D$ be the discriminant of a real quadratic field of the form
$D=4E$ or $D=E$ with $E$ square free.  Let $C$ be a primitive
Teichm\"uller curve in genus $g=2$ of discriminant $D$. Suppose that
the affine group, $\Gamma$, of $C$ is commensurable to a triangle
group. Then $C=W_D^\varepsilon$ with $D\leq 12$.
\end{Lemma}

\begin{proof}
Let $D$ be as in the statement of the lemma.  Suppose that the affine
group of $W_D^\varepsilon$ is commensurable to a triangle group
$\Delta(n,m,\infty)$. Commensurable Fuchsian groups have the same
trace field, therefore $\QQ(\sqrt{D})=\QQ(\zeta_n+\zeta_n^{-1},
\zeta_m+\zeta_m^{-1})$.  This is only possible for $D\in \{5, 8,
12\}$.
\end{proof}

\begin{Rem} It is an open question whether there are infinitely many primitive 
Teichm\"uller curves in genus $2$ whose affine group is commensurable
 to a triangle group.  If $D$ is the discriminant of such a
 Teichm\"uller curve, then Lemma \ref{trianglelem} implies that
 $D=4EF^2$ or $EF^2$ with $E$ square free and $D/F^2\leq 12$.
\end{Rem}

\section{Generalities on differential equations}\label{pcurvsec}
In this section we recall some generalities on flat vector
bundles. Among other things we recall the connection between flat
vector bundles and differential equations, and define global
nilpotence.

Let $(\cEE, \nabla)$ be a flat vector bundle of rank $2$ on a smooth
projective curve $\ol{C}$ with regular singularities in $\tau_1,
\ldots, \tau_r\in \ol{C}$. In the case that $\ol{C}$ is a
Teichm\"uller curve, the set of singularities is the union of the
cusps with the elliptic points. Let $t$ be a local parameter of
$\ol{C}$ at $\tau_i$, and let ${\mathfrak m}_{\tau_i}$ the maximal
ideal of the local ring ${\mathcal O}_{\ol{C}, \tau_i}$. The {\em
monodromy operator} $\mu_i$ is the endomorphism of the fiber
$\cEE|_{\tau_i}=\cEE_{\tau_i}/{\mathfrak m}_{\tau_i}\cdot
\cEE_{\tau_i}$ defined by $\nabla(t\partial/\partial t)$. One checks
that $\mu_i$ does not depend on the choice of $t$.

\begin{Defi}
The {\em local exponents} $\gamma_{1,\tau_i}, \gamma_{2, \tau_i}$ of
$\cEE$ at $\tau_i$ are the eigenvalues of $\mu_i$.
\end{Defi}

 If $\tau_i$ is a cusp then $\mu_i$ is quasi-unipotent. This
implies that $\gamma_{1, \tau_i}\equiv \gamma_{2, \tau_i} \pmod{\ZZ}$.

In the case that $g(\ol{C})=0$, a flat vector bundle $\cEE$ of rank
$2$ corresponds to a differential equation, as follows. Let $t$ be a
parameter of $\ol{C}$.  Choose a (rational) section $s$ of $\cEE$ such
that $s$ and $s':=\partial s/\partial t$ are generically
independent. Such a section is called a {\em cyclic vector}. Since the
rank of $\cEE$ is $2$, there exist rational functions $p_1, p_2\in
k(x)$ such that
\[
s''+p_1 s'+p_2s=0.
\]
The differential operator $L=(\partial/\partial
t)^2+p_1(\partial/\partial t)+p_2$ is Fuchsian. Its  singularities are the
singularities of $\nabla$, and possibly $t=\infty$. Recall that this
means that $\ord_{x_j} p_i\geq -i$. The differential operator $L$ is
called the {\em differential operator} associated with $(\cEE, s)$. This
defines an equivalence between differential equations and vector
bundles with a section.

If $\cEE$ possesses a nontrivial Hodge filtration $\cLL\subset \cEE$,
for example if $\cEE$ is indigenous, we always choose $s$ to be a
rational section of $\cLL$. Replacing $s$ by a multiple changes $L$ to
an equivalent differential operator. If $\cEE$ is indigenous the
differential operator is just the uniformizing differential equation
corresponding to $\HH\to \HH/\Gamma=C$, as defined in \cite{Yoshida},
\S~5.2. This may be seen by remarking that the affine group $\Gamma$
is the monodromy group of $L$.

The notion of local exponents we defined above agrees with the
classical notion of local exponents of a differential operator. Write
$p_j=\sum_{n\geq -j} c_{j, n}t^n$. Then the local exponents are the
roots of the indicial equation $t(t-1)+tp_{1, -1}+p_{2, -2}=0$.

If $L$ is hypergeometric, i.e.\ if $L$ has exactly three singularities
which we may suppose to be $\{0,1,\infty\}$, $L$ is determined by its
local exponents. This is no longer the case if $L$ has more than three
singularities. Namely, apart from the position of the singularities
and the local exponents, $L$ also depends on the so called {\em
accessory parameters}. The main problem in determining the
differential equation corresponding to $W_D^\varepsilon$ explicitly in
\S \ref{desec} is to determine the accessory parameters.

\bigskip\noindent We now recall from \cite{Katz} the notion of a {
globally nilpotent differential operator}. Let $R$ be an integral domain 
which is finitely
generated (as ring) over $\ZZ$, and whose fraction field $\Frac(R)$
has characteristic zero. Let $\pi:\ol{\mathcal C}\to \Spec(R)$ be a
smooth morphism of relative dimension $1$. (In our case, $R$ is
an order in an number field with finitely many primes inverted, and
$\ol{\mathcal C}$ will be a model of the Teichm\"uller curve $\ol{C}$
over $R$.)  Let ${\mathcal D}$ be an \'etale divisor on $\ol{\mathcal
C}$.  We let $(\cEE, \nabla)$ be a flat vector bundle of rank $2$ on
$\ol{C}=\ol{\mathcal C}\otimes_R \Frac(R)$ with regular singularities
in $D={\mathcal D}\otimes_R \Frac(R)$. Let $p$ be a prime number which
is not invertible on $\ol{\mathcal C}$, and let $\wp$ be a prime ideal
of $R$ above $p$. Write $\FF_q=R/\wp$. Reduction modulo $p$ defines a
flat vector bundle $\cEE_\wp:=\cEE\otimes \FF_q$ on
$\ol{C}_\wp:=\ol{\mathcal C}\otimes \FF_q$.

Define ${\mathcal T}:=(\Omega_{\ol{C}_\wp/k}(\log{\mathcal D}\otimes
k))^{\otimes -p}$. The $p$-{\em curvature} of $\cEE_\wp$ is an ${\mathcal
O}_{\ol{C}_\wp}$-linear morphism
$\Psi_{\cEE}:{\mathcal T}\to \End_{{\mathcal
O}_{\ol{C}_\wp}} (\cEE)$ defined by
\begin{equation}\label{pcurveq}
\Psi_{\cEE}(D^{\otimes p}):=\nabla(D)^{\otimes p}-\nabla(D^{\otimes p}).
\end{equation}
For details, we refer to \cite{Katz}, \S~5.

\begin{Defi}\label{pcurvdef}
\begin{itemize}
\item[(a)] The $p$-curvature of $\cEE_\wp$ is {\em nilpotent} if
$\Psi_{\cEE_\wp}$ consists of nilpotent endomorphisms.
\item[(b)] The flat vector bundle $\cEE$ is {\em globally nilpotent}
if the $p$-curvature $\cEE_\wp$ is nilpotent for all but finitely many $\wp$.
\item[(c)] If $L$ is a differential operator of order $2$, we say that
$L$ has {\em nilpotent $p$-curvature} if the $p$-curvature of the
corresponding flat vector bundle is nilpotent.
\end{itemize}
\end{Defi}

Katz' Theorem (\cite{Katz}, Theorem 10.0) states that a flat vector
bundle $(\cEE, \nabla)$ which is the direct factor of the relative de
Rham cohomology of a family of curves defined over a number field is
globally nilpotent.  The local exponents of globally nilpotent
flat vector bundles are rational numbers.

If $g(\ol{C})=0$, one may rephrase the
notion of nilpotent $p$-curvature, as follows (\cite{Honda}).  We let
$\wp|p$ be invertible on $\ol{\mathcal C}$, and write $L_\wp$ for the
differential operator corresponding to $\cEE_\wp$. Since the local
exponents of $L$ are rational numbers in our situation, the
$p$-curvature of $\cEE_\wp$ is nilpotent if and only if $L_\wp$ has a
polynomial solution.

The $p$-curvature of $\cEE_\wp$ is identically zero if $\Psi_{\cEE_\wp}$
is zero. If $g(C)=0$, this is equivalent to saying that the
corresponding differential operator $L_\wp$ has a basis of polynomial
solutions. Since Teichm\"uller curves always have a cusp, it follows
from \cite{Honda}, Proposition 5.1, that the $p$-curvature of
the differential operators we consider do not have zero $p$-curvature. 

The analog of this notion in characteristic zero is that $L$ has a
basis of algebraic solutions. Here $u\in \CC(\!(t)\!)$ is called {\em
algebraic} is if it algebraic over $\CC(t)$. The differential operator
$L$ has a basis of algebraic solutions if and only if its monodromy
group is finite. Since the monodromy group of the uniformizing
differential equation $L$ of a Teichm\"uller curves is never finite,
$L$ does not have a basis of algebraic solutions.

We state a well-known result on solutions of differential
operators in characteristic zero.  It follows
immediately from Fuchs' Theorem (\cite{Beukers}, Theorem 2.9). Note
that if $\tau$ is a cusp of $L$ then after replacing $\ol{C}$ by a
finite cover and $L$ by an equivalent differential operator, we may
assume that the local exponents of $L$ at $\tau$ are
$(0,0)$. Therefore the condition in Lemma \ref{Fuchslem} is no serious
restriction.

\begin{Lemma}\label{Fuchslem}
Let $L$ be a Fuchsian differential operator of order $2$. Suppose that
$\tau$ is a cusp, and let $t$ be a local parameter in $\tau$. Assume
that the local exponents of $L$ at $\tau$ are $(0,0)$.  Then there  exists a 
unique solution $u\in \CC[[t]]$ of $L$ around $t=0$ with the
property that $u(x=0)=1$.
\end{Lemma}

\section{Counterexamples to the conjectures of Dwork and 
Chudnovsky--Chudnovsky}
\label{counterexasec}
In this section we use the results from the previous sections to show
that the $W_D^\varepsilon$'s provide counterexamples to conjectures of
Chudnovsky--Chudnovsky and Dwork. The main ingredients of the proofs
can already be found in \cite{McM03}. Our contribution here is
linking the theory of Teichm\"uller curves with that of differential
equations. In \S~\ref{desec} we then find an explicit formula for some
of the differential equations. This gives then also explicit
counterexamples to these conjectures.

The following conjecture is stated by Chudnovsky and Chudnovsky in
\cite{Chud}. The original conjecture is stated in the language of
differential operators, i.e.\ Chudnovsky--Chudnovsky assume that the
genus of $\ol{C}$ is $0$. But there is no need for this restriction
(as both variants of the conjecture are wrong).

\begin{Conj}[Chudnovsky--Chudnovsky]\label{Chudconj}
Let $\ol{C}$ be a smooth projective curve defined over a number field.
Let $(\cEE, \nabla)$ be an indigenous bundle on $\ol{C}$, and let
$\Gamma\subset \PSL_2(\RR)$ be the monodromy group of $\cEE$.  Suppose
that $\cEE$ is globally nilpotent. Then $\Gamma$ is either arithmetic
or commensurable to a triangle group.
\end{Conj}

We now return to the situation that $C=\HH/\Gamma$ is a primitive
Teichm\"uller curve in genus $2$, of discriminant $D$. We suppose that
the affine group $\Gamma$ of $C$ is not commensurable to a triangle
group (Lemma \ref{trianglelem}). Recall that
$H^1_\dR(\cXX)=\cEE_1\oplus \cEE_2$ is a decomposition of flat vector
bundles of rank $2$ (\S~\ref{Teichsec}). We remark that
Proposition~\ref{Chudprop} holds for all Teichm\"uller curves whose
affine group is not commensurable to a triangle group. One could
extend the result to the Teichm\"uller curves in genus $g=3,4$ which
where found by McMullen. Proposition \ref{Chudprop}, together with
Lemma \ref{trianglelem}, produces counterexamples to Conjecture
\ref{Chudconj}.

\begin{Prop}\label{Chudprop}
\begin{itemize}
\item[(a)] The flat vector bundles $\cEE_i$ are globally nilpotent.
\item[(b)] Suppose that $D\geq 13$. Then $\Gamma$ is nonarithmetic.
\end{itemize}
\end{Prop}

\begin{proof}
Part (a) follows from the theorem of Katz (\cite{Katz}, Theorem
10.0). Lemma \ref{trianglelem} implies that $\Gamma$ is not
commensurable to a triangle group. Since $C$ is a primitive
Teichm\"uller curve in genus $g=2$, the affine group $\Gamma$ is not arithmetic
(\cite{Mo06a}, Corollary~2.10.)
\end{proof}

Note that we do not know a priori which are the finitely many values
$p$ for which the $p$-curvature of $\cEE_i$ is not nilpotent. For the 
families of
curves corresponding to $W_{13}$ and $W_{17}^\varepsilon$ we answer
this question in Proposition~\ref{reductionprop}.

In the literature one also finds variants of Conjecture
\ref{Chudconj}, omitting either the condition that $\ol{C}$ is defined
over a number field or that $\cEE$ is globally nilpotent. Proposition
\ref{Chudprop} shows that these variants also do not hold.

We now turn to Dwork's Conjecture (\cite{Dwork}, Conjecture
7.4).

\begin{Conj}[Dwork] \label{Dworkconj}
Let $\ol{C}$ be a smooth projective curve of genus $0$ defined over a
number field, and let $(\cEE, \nabla)$ be a flat vector bundle of rank
$2$ which is globally nilpotent. Then either the monodromy group of
$\cEE$ is commensurable to a triangle group or $\cEE$ has an algebraic
solution.
\end{Conj}

\begin{Prop}\label{Dworkprop}
Let $C$ be a primitive Teichm\"uller curve in genus $2$ of
discriminant $D\geq 13$. Suppose that $g(C)=0$. Then the flat vector
bundle $\cEE_1$ does not admit an algebraic solution.
\end{Prop}

\begin{proof} Let $L_1$ be the differential operator corresponding to $\cEE_1$.
Let $\{\tau_1, \ldots, \tau_r\}$ be the set of singularities, and let
$\xi\in \ol{C}\setminus\{\tau_i\}$ be a base point. Analytic
continuation defines the monodromy representation
\[
\rho:\pi_1(\ol{C}\setminus\{\tau_i\}, \xi)\to \PGL_2(\CC),
\]
which has image $\Gamma$. Since $\Gamma$ is the affine group of a
Teichm\"uller curve, it has at least one cusp. In particular, $\Gamma$
is not finite. This implies that $\Gamma$ does not have a basis of
algebraic solutions.

Suppose that $L_1$ has an algebraic solution $u$. For every loop
$\gamma$ in $\pi_1(\ol{C}\setminus\{\tau_i\}, \xi)$, the image of $u$
under the monodromy matrix $\rho(\gamma)$ is again an algebraic
solution of $L_1$. Since $L_1$ does not have a basis of solutions,
this solution is again $u$ (up to a nonzero constant). This implies
that $u$ is in fact a rational function.  Since the genus of $C$ is
zero, we find a contradiction.
\end{proof}

 McMullen (\cite{McM03}, Theorem 9.8) shows that the Teichm\"uller
curves $W_D^\varepsilon$ for discriminant $D\in \{5, 8, 12, 13, 17,
21, 28, 29, 33\}$ are rational. Therefore Proposition \ref{Dworkprop}
is a nonempty statement. In particular, together with Lemma
\ref{trianglelem}, it also follows that there exist differential
equations whose monodromy group is nonarithmetic which are
counterexamples to Dwork's Conjecture.

\begin{Rem}\label{Daanrem}
An arithmetic counterexample to Dwork's conjecture has previously been
found by Krammer (\cite{Krammer}). Krammer considers an arithmetic
Fuchsian group $\Gamma\subset \SL_2(\RR)$ which is not commensurable
to a triangle group such that $C=\HH/\Gamma$ is a Shimura curve. 
To compute the uniformizing differential equation $L$, Krammer
finds a subgroup of finite
index of $\Gamma$ of the form $\Gamma':=\Gamma\cap g\Gamma g^{-1}$,
where $g$ is an element of the commensurator of $\Gamma$. This yields
a correspondence $\HH/\Gamma' \rightrightarrows C$ which allows him to
determine $L$.  A similar method for computing the uniformizing
differential equation of a Shimura curve has also been used by Elkies
(\cite{Elkies}).

In Krammer's approach it is essential to consider an arithmetic
lattice $\Gamma$. Namely, for nonarithmetic lattices $\Gamma$ the
subgroup $\Gamma\cap g\Gamma g^{-1}$ is unlikely to have finite index,
since in this case the commensurator has finite index in $\Gamma$
(\cite{Margulis}, Theorem B). The reason why it is hard to find
counterexamples to the conjecture of Chudnovsky--Chudnovsky appears to
be that it is difficult to find differential equations which
correspond to a nonarithmetic groups which is not a triangle group.
\end{Rem}

\section{An equation for the Teichm\"uller curves $W_{13}$ and 
$W_{17}^\varepsilon$}
\label{eqnsec}
In the section we explicitly calculate an equation for the family of
curves corresponding to the Teichm\"uller curves $W_{13}$ and
$W^\ve_{17}$. We give the proof only for $D=17$, and leave the case
$D=13$ to the reader.  We work over a double cover $\ol{C} \to
\ol{W}^\ve_{17} $. We start by assuming that there exists a universal
family $f:\ol{\cXX}\to \ol{C}$. In Lemma \ref{necclem}, we derive
necessary conditions for the family $\ol{\cXX}$. We then show that
there is a (in fact unique) family of curves satisfying these
conditions. We  deduce that this family is  the universal
family corresponding to the Teichm\"uller $W_{17}^\varepsilon$.

Using the algorithm of \cite{McM03}, Theorem~9.8, one finds that the
Teichm\"uller curve $\ol{W}^\ve_{17}$ is a projective line with $3$
cusps and one elliptic point of order two. We may choose a parameter
$s$ of $\ol{W}_{17}^\varepsilon$ such that the elliptic point is $s=-1$ and
such that $s=1, \infty$ are cusps. 
We consider a cover $\pi:\ol{C}\to
\ol{W}_{17}^\varepsilon$ of degree $2$ which is branched at $s=-1,
\infty$. Then $\ol{C}$ is also a projective line; we may choose a
coordinate $t$ of $\ol{C}$ such that $s=(t^2+1)/2t$. It follows that
the set of cusps on $\ol{C}$ is $S=\{0, 1, \infty, \tau, 1/\tau\}$, for
some $\tau\in \RR$ (in fact, we will see that $\tau\in \QQ(\sqrt{17})$).

Suppose there exists a family $\ol{f}: \ol{\cXX} \to \ol{C}$ of stable
curves of genus $2$ such that the moduli map $\ol{C} \to \ol{\cMM}_2$
factors through $\ol{W}_{17}^\varepsilon$. Since
$\deg(\Omega^1_{\ol{C}}(\log S))=|S|-2=3$ is odd, we need to replace
$\ol{C}$ by a cover of degree $2$ the define the bundles $\cLL_i$
(Remark \ref{indicritrem}).  Let
$\tilde{C}\to \ol{C}$ be the cover of degree $2$ which is branched at
$t=0, \infty$, and let $\tilde{S}\subset \tilde{C}$ be the inverse
image of $S$. On $\tilde{C}$ we have that
$2\deg(\tilde{\cLL}_1)=-2\deg(\tilde{\cEE}_1/\tilde{\cLL}_1)=
\deg(\Omega^1_{\tilde{C}}(\log \tilde{S}))=|\tilde{S}|-2=6$. We
conclude that $\deg(\tilde{\cLL}_1)=3$ and
$\deg(\tilde{\cLL}_2)=\deg(\tilde{\cLL}_1)/3=1$ (Corollary
\ref{lyacor}). Since in the sequel only $\cLL_1^i\otimes \cLL_2^j$ with
$i+j$ even appear, it suffices to work on $\ol{C}$.

We use the notation of \S~\ref{lyasec}.  In particular, we let
$\delta(k):\cLL_1\otimes\cLL_2\to \cLL_1^k\otimes\cLL_2^{6-k}$ be the
maps defined in Proposition \ref{deltaprop}. Let $s_i$ be a
holomorphic section of $\cLL_i$ which only has zeros in $t=0$. (These
sections are unique up to a $\CC^\times$-multiple.) As in Proposition
\ref{deltaprop}, we obtain functions $c_k$ such that $\ol{\cXX}$ is
given by the equation
\begin{equation}\label{eqneq}
y^2=g_{17}(x):=\sum_{k=0}^5 c_k x^i, \qquad \text{where }x=s_1/s_2. 
\end{equation}
The calculation of $\deg{\tilde{\cLL}_i}$ implies  that
$\deg_t(c_k)=5-k$.

The following lemma follows immediately from the classification of the
cusps (\S~\ref{protosec}).

\begin{Lemma}\label{necclem}
Suppose that there exists a family of stable curves $\ol{f}: \ol{\cXX} \to
\ol{C}$ of genus $2$ such that the moduli map $\ol{C} \to \ol{\cMM}_2$
factors through $\ol{W}_D^\varepsilon$.

We may choose a coordinate $t$ on $\ol{C} \sms \infty$ such that, if 
$s_i$ is the unique holomorphic section of $\cLL_i$ up to $\CC^*$
with zeros concentrated at $t=0$, then $c_{6-k}$ as in equation
\eqref{HENF} is  a polynomial of degree $k-1$ and satisfies
$$ c_k(1/t)= t^{5-k}c_k(t).$$ 
Then the following holds:
\begin{itemize}
\item[(a)] the fiber over $t=0$ has an equation \eqref{normalformeq} (up
to rescaling $x$ to $\rho_0 x$) for the
splitting prototype $(0,4,1,1)$ if $\ve=1$ and  $(0,4,1,-1)$ if $\ve=0$,
\item[(b)] the fiber over $t=1$ has an equation \eqref{normalformeq} (up
to rescaling $x$ to $\rho_1 x$) for the
splitting prototype $(1,2,2,-1)$ if $\ve=1$ and  $(0,2,2,-1)$ if $\ve=0$,
\item[(c)] there is a $\lambda \in \RR$ such that the fiber over $t=\lambda$
has an equation \eqref{normalformeq} (up
to rescaling $x$ to $\rho_\lambda x$) for the
splitting prototype $(0,2,1,-3)$ if $\ve=1$ and  $(0,1,2,-3)$ if $\ve=0$.
\end{itemize}
\end{Lemma}

We first consider the case that $\varepsilon=1$. Define
\begin{equation}\label{c17eq}
\begin{split}
c_0&=(425765\sqrt{17}+1755475)t^5+(-5289173-1282803\sqrt{17})t^4+\\
&\qquad +(857038 \sqrt{17}+3533762)t^3+
(857038\sqrt{17}+3533762)t^2+\\ &\qquad +(-5289173-1282803\sqrt{17})
t+425765\sqrt{17}+1755475\\
c_1&=\frac{1}{2}((80325\sqrt{17}+331187)t^4+(-310964\sqrt{17}-1281964)t^3+\\
&\qquad + (461278\sqrt{17}+1901714)t^2
+(-310964\sqrt{17}-1281964)t+\\&\qquad +80325\sqrt{17}+ 331187),\\
c_2&=\frac{1}{4}((-3825\sqrt{17}-15783)t^3+(3825\sqrt{17}+15687)t^2+(3825
\sqrt{17}+15687)t+\\ &\qquad +3825\sqrt{17}-15783),\\
c_3&=\frac{1}{8}((-1105\sqrt{17}-4551)t^2
+(2210\sqrt{17}+8918)t-1105\sqrt{17} -4551),\\ c_4&=3+3t, \qquad\qquad
c_5=1.
\end{split}
\end{equation}

The following proposition follows by straightforward verification. 

\begin{Prop}\label{degenprop}
The family of stable curves $\ol{f}:\ol{\cXX}\to \ol{C}$ defined by
 (\ref{eqneq}, \ref{c17eq}) is the unique family satisfying the
 conditions of Lemma \ref{necclem}.
\end{Prop}

The curve for $t=-1$ is isomorphic to
$$y^2= x(x^2-1)(x-a)(x-1/a),
\quad \text{where} \quad a=20+5\sqrt{17}+2\sqrt{206+50\sqrt{17}},$$
which can been calculated by Silhol's algorithm (\cite{Silhol}).

 One may deduce from the uniqueness in Proposition \ref{degenprop}
that (\ref{eqneq}) is the universal family of curves corresponding to
$W_{17}^\varepsilon$.  In \S~\ref{desec}, we follow an alternative
approach to prove this. Namely, we show the existence of an indigenous
bundle on $\ol{C}$.

 It suffices to consider  $\varepsilon=1$. One
obtains an equation for the universal family corresponding to
$W_{17}^0$ by replacing $\sqrt{17}$ by $-\sqrt{17}$ (Theorem
\ref{Galoisthm}).

\bigskip\noindent
We now define a family of curves which will turn out to be the family
of curves corresponding to $W_{13}$. One may formulate the analog of
Lemma \ref{necclem} also in this case, and show that the family of
curves below satisfies these necessary conditions. 

Using the algorithm in \cite{McM03}, Theorem 9.8, one finds that the
 Teichm\"uller curve $\ol{W}_{13}$ is a projective curve with three
 cusps and an elliptic point of order $2$. As in the case that $D=17$,
 we let $\ol{C}_{13}\to \ol{W}_{13}$ be a cover which branches at the
 elliptic point and one of the cusps. On $\ol{C}_{13}$ we then have a
 set, $S_{13}$, of $5$ cusps. We may choose a parameter $t$ of
 $\ol{C}_{13}$ such that $S_{13}=\{0, 1, \infty, \rho, \rho^{-1}\}$.
 We consider a family $\ol{\cXX}_{13}$ of stable curve of genus $2$
 defined by
\begin{equation}\label{X13eq}
y^2=g_{13}(x)=\sum_{k=0}^5 c_k x^k
\end{equation}
where
\begin{equation}\label{c13eq}
\begin{split}
c_0&=\frac{1}{2^{15}}((-1585778688\sqrt{13}+5717606400)t^5+(4758908544\sqrt{13}
-
17158488768)t^4+\\
&\qquad+(11440882287-3173129856\sqrt{13})t^3+(11440882287-
3173129856\sqrt{13})t^2+\\ &\qquad+
(4758908544\sqrt{13}-17158488768)t-1585778688\sqrt{13}+ 5717606400),\\
c_1&=\frac{1}{2^{12}}((-31629312\sqrt{13}+114041088)t^4+(124405632\sqrt{13}-
448550784)t^3+\\ &\qquad
+(-185552640\sqrt{13}+669019797)t^2+(124405632\sqrt{13}-448550784)t+\\
&\qquad -31629312 \sqrt{13}+114041088),\\
c_2&=\frac{1}{2^{5}}((7488\sqrt{13}-27000)t^3+(-7488\sqrt{13}+26991)t^2+(-7488
\sqrt{13}+
26991)t+\\ &\qquad + 7488\sqrt{13}-27000),\\
c_3&=\frac{1}{2^{6}}((-14992+4160\sqrt{13})t^2+(30011-8320\sqrt{13})t-14992+
4160
\sqrt{13})\\ c_4&=t+1,\qquad\qquad c_5=1.
\end{split}
\end{equation}

The following proposition is proved by a straightforward computation,
similar to the proof of Proposition \ref{degenprop}.

\begin{Prop}\label{degenprop13}
The family of stable curves $\ol{f}:\ol{\cXX}\to \ol{C}$ defined by
 (\ref{X13eq}, \ref{c13eq}) is the unique family satisfying the analog
 of Lemma \ref{necclem} for $D=13$.
\end{Prop}

\section{The differential equation associated with the Teichm\"uller curves 
$W_{17}^\varepsilon$ and $W_{13}$}\label{desec} The goal of this
section is to show that the family $\ol{\cXX}^\varepsilon(D)$ of
curves from \S~\ref{eqnsec} is the universal family corresponding to
the Teichm\"uller curve $\ol{W}_{D}^\varepsilon$. The key step is
showing the existence of an indigenous bundle on a finite cover of
$\ol{W}_{D}^\varepsilon$. We only treat the case $D=17$ and
$\varepsilon=1$ in detail. The proof for $D=17$ and
$\varepsilon=0$ is obtained by replacing $\sqrt{17}$ by
$-\sqrt{17}$. The case $D=13$ is similar, and left to the reader.

We let $D=17$ and $\ve=1$, and drop $D$ and $\ve$ from the notation.
 Let $K=\QQ(\sqrt{17})$. Let $\ol{f}:\ol{\cXX}\to \ol{C}$ be the
 family of stable curves given by (\ref{eqneq}, \ref{c17eq}) over
 $K$. The differential forms
\[
\omega_1:=\frac{{\rm d}x}{y}, \qquad \omega_2:=\frac{x\,{\rm d}x}{y}
\]
form a basis of $H^0(\ol{\cXX}, \Omega^1_{\ol{C}/K})$. More precisely,
$\omega_i$ is a section of $\cLL_i=\cEE_i^{(1,0)}$.  Recall that 
$S=\{0,1,\infty, \tau, \tau^{-1}\}$,
where $\tau$ is defined by (\ref{taueq}). Write $S'=S\setminus
\{\infty\}$. Let $$\nabla:
\cHH^1_\dR(\ol{\cXX})\to
\cHH^1_\dR(\ol{\cXX})\otimes\Omega^1_{\ol{C}/K}(\log S)$$ be 
the Gau\ss--Manin connection, and let $t$ be the parameter of 
$\ol{C}$ chosen in \S~\ref{eqnsec}. We write 
$\omega'=\nabla(\partial/\partial t)\omega_i$.

\begin{Prop}\label{deprop}
The sections $\omega_i$ attached
to $W_{17}^1$ satisfy the differential equation
\begin{equation}\label{de17eq}
\omega_i''+A_i(t)\omega_i'+B_i(t)\omega_i=0\in \cHH^1_\dR(\ol{\cXX})
\end{equation}
with
\[
A_1=\sum_{\tau_i\in S'} \frac{1}{t-\tau_i}, \qquad
B_1=\frac{576t^2 + (-4557 + 891\sqrt{17})t + (1296 -
240\sqrt{17})}{2^8\prod_{\tau_i\in S'}(t-\tau_i)}
\]
and
\[
A_2=\sum_{\tau_i\in S'} \frac{1}{t-\tau_i}
-\frac{1}{t-\mu_1}-\frac{1}{t-\mu_2}, \qquad B_2=\frac{{\mathcal B}}
{2^{14}(t-\mu_1)(t-\mu_2)
\prod_{\tau_i\in
    S'} (t-\tau_i)},
\]
where
\[\begin{split}
{\mathcal B}&=4096t^4+(35616-12768\sqrt{17})t^3+(-260375+69633\sqrt{17})t^2+\\
&\qquad
 + (-177536+38528\sqrt{17})t+47104-10240\sqrt{17},
\end{split}
\]
and where $\mu_1$ and $\mu_2$ are the roots of
$128t^2+(137-95\sqrt{17})t+128=0$.
\end{Prop}

\begin{proof}
A straightforward, but somewhat tedious computation shows that
$\omega_i''+A_i(t)\omega_i'+B_i(t)\omega_i$ is an exact differential
form, for $i=1,2$. 
\end{proof}

Proposition \ref{deprop} implies that there exists a decomposition
$\cHH^1_\dR(\ol{\cXX})=\cEE_1\oplus\cEE_2$ of flat vector bundles. We
define line bundles $\cLL_i$ on $\ol{C}$ by $\cLL_i=\cEE_i^{(1,0)}$. By
definition, $\omega_i$ is a section of $\cLL_i$.

\begin{Lemma}\label{KSlem}
 The Kodaira--Spencer map $\Theta:\cLL_1\to (\cEE_1/\cLL_1)\otimes
\Omega^1_{\ol{C}}(\log S)$ is an isomorphism, i.e.\ $\cEE_1$ is an
indigenous bundle.
\end{Lemma}

\begin{proof} 
At points of $C$, i.e.\ where the fiber of $f$ is smooth, the
differential equation~\eqref{de17eq} has no singularities. Consequently,
the Kodaira--Spencer map does not vanish at those points (\cite{BoMo05} 
Proposition~2.2(b)). The nonvanishing of the Kodaira--Spencer
map at points in $S$ is shown in (\cite{BoMo05}, Proposition~2.2(d)).
\end{proof}

\begin{Prop}\label{indiprop}
The family $\ol{\cXX}$ of stable curves is the universal family
corresponding to the Teichm\"uller curve $W_{17}^1$.
\end{Prop}

\begin{proof} Since $\cEE_1$ is an indigenous bundle on $\ol{C}$, it follows 
from \cite{Mo06a}, Proposition 5.5, that $\ol{C}$ is the cover of a
Teichm\"uller curve. Since the moduli map $\ol{C}\to \ol{\cMM}_2$
factors through $\ol{W}_{17}^1$, the proposition follows.
\end{proof}

Recall that $\ol{C}\to \ol{W}_{17}^\varepsilon$ is a degree-$2$ cover
given by $s=(t^2+1)/2t$. We denote by $\Sigma=\{1, -1, \infty,
(\tau^2+1)/2\tau\}$ the union of the set of cusps with the
elliptic point on $\ol{W}_{17}^\varepsilon$. 

\begin{Cor}\label{fomcor}
The pointed curve $(\ol{W}_{17}^\varepsilon, \Sigma)$ may not be defined
over $\QQ$.
\end{Cor}

\begin{proof} It suffices to check that  the $j$-invariant of 
$\Sigma$ is not in $\QQ$.
\end{proof}

We finish this section by giving a formula for the differential
equations corresponding to $W_{13}$. Recall that the set of cusps is in this
case $S=\{0, 1, \infty, \rho, \rho^{-1}\}$, where $\rho$ is given by
(\ref{rhoeq}). Let $S'=S\setminus \{\infty\}$. Let $\mu_1, \mu_2$ be
the roots of $t^2+(5+28\sqrt{13})t/48+ 1=0$.
\par
\begin{Prop}\label{de13prop}
The sections $\omega_i$ attached
to $W_{13}$ satisfy the differential equation
\begin{equation}\label{deeq}
\omega_i''+A_i(t)\omega_i'+B_i(t)\omega_i=0\in \cHH^1_\dR(\ol{\cXX})
\end{equation}
with 
\[
A_1=\sum_{\tau_i\in S'}\frac{1}{t-\tau_i}, \qquad
B_1=\frac{1152t^2+(333\sqrt{13}-576)t-120\sqrt{13}+192}{2^9
\prod_{\tau_i\in S'}(t-\tau_i)}
\]
and
\[
\begin{split}
A_2&=\sum_{\tau_i\in
S'}\frac{1}{t-\tau_i}-\frac{1}{t-\mu_1}-\frac{1}{t-\mu_2} ,\\
B_2&=\frac{98304t^4+(135936\sqrt{13}+67584)t^3+(698944+5744\sqrt{13})t^2-80181+
21234\sqrt{13}}{2^{17}3(t-\mu_1)(t-\mu_2)\prod_{t\in
S'}(t-\tau_i)}.
\end{split}
\]
\end{Prop}
\par
The proofs of Lemma \ref{KSlem} and Proposition \ref{indiprop}
immediately carry over to this situation, and we conclude that
(\ref{X13eq}, \ref{c13eq}) define the universal family of stable curves
of genus $2$ corresponding to the Teichm\"uller curve $W_{13}$.

\begin{Rem}  Let $\Sigma=\{\pm 1, \infty, (\rho^2+1)/2\rho \}\subset 
\ol{W}_{13}$.
 Then the $j$-invariant of $\Sigma$ is indeed in $\QQ$, as also
follows from Theorem \ref{Galoisthm}.
\end{Rem}

\section{Integral solutions of  differential equations}
\label{solsec}\label{integralsec}

The expansions of the solution $u_i$ of the
differential operators $L_i$ from \eqref{de17eq} at $t=0$ are given by
\begin{equation}
\begin{split}
u_1 & = 1+ \frac{81 - 15\sqrt{17}}{2^4}t + 
\frac{4845- 1155\sqrt{17}}{2^6}t^2 
+ \frac{3200225 - 775495\sqrt{17}}{2^{11}}t^3 +\cdots, \\
u_2 & = 1+ \frac{ 23- 5\sqrt{17}}{2^3}t + \frac{5561-1343\sqrt{17}}{2^7}t^2
+\frac{452759- 109793\sqrt{17}}{2^9}t^3+\cdots
\end{split}
\end{equation}
On the other hand, the coefficients of $u_i$ are determined by
the recursion formula \eqref{receq} below which causes a division
by $(j+1)^2$ in the $j$-th step. The purpose of this section
is to explain why nevertheless almost no denominators occur.
This rare phenomenon also occurs in Ap\'ery's differential
equation (\cite{BS}). In Ap\'ery's case the integrality of the
coefficients is shown by guessing a closed formula. This seems
much harder in our case. Instead we show that 
$L_i$ admits an integral solution
$\pmod{\wp^n}$ for all $n$.  This relies on a result of Katz
(\cite{Katz84}) which states that the expansion coefficients of
$\omega_i$ are solutions of $L_i$ modulo $p^n$.
\par
Our strategy works more generally. We fix $D$ and 
$\varepsilon\in \{0, 1\}$, and let
$\ol{W}_D^\varepsilon$ be the corresponding Teichm\"uller curve. We
exclude the Teichm\"uller curve corresponding to $(X, \omega)\in
\Omega\cMM_2(1,1)$. This is no restriction, as the affine group in
this case is a triangle group, and the results of this section are
well-understood in that case. 

Let $\ol{C}\to \ol{W}_D^\varepsilon$ be a
cover which is only branched at the cusps and elliptic points of
$\ol{W}_D^\varepsilon$, such that $\ol{C}$ does not have any elliptic
points. Such a cover always exists. In the case that $D=13, 17$, the
cover $\ol{C}\to \ol{W}_D^\varepsilon$ is the degree-$2$ cover of \S
\ref{eqnsec}. In this section, we consider the case that
$g(\ol{C})=0$.

The assumption that $g(\ol{C})=0$ implies that there exist
differential operators $L_i$ on $\ol{C}$ corresponding to the flat
vector bundles $\cEE_i$ on $\ol{C}$ (compare to the discussion in
\S~\ref{pcurvsec}). In the case that $D=13, 17$ these are the
differential operators determined in \S~\ref{desec}. The singularities
of $L_1$ are exactly the cusps of $\ol{C}$. Moreover, since the genus
of $\ol{C}$ is zero, we may choose a coordinate $t$ on $\ol{C}$ such
that $t=0, \infty$ are cusps. The singularities of $L_2$ are the cusps
together with the zeros of the Kodaira--Spencer morphism of $\cEE_2$.

We assume, moreover, that the local exponents of $L_i$ at all
singularities, except possibly $t=\infty$, are integers. After
replacing $L_i$ by an equivalent differential operator, we may
therefore assume that the local exponents of $L_i$ at the cusps $t\neq
\infty$ are $(0,0)$, and that the local exponents of $L_2$ at the
zeros of the Kodaira--Spencer map are $(0, \delta_j)$, with
$\delta_j\geq 0$ (in Lemma \ref{locexplem}, we give a more precise
statement).  In the case that $D=13, 17$ the differential operators
$L_i$ computed in \S~\ref{desec} satisfy these conditions. Moreover,
one checks using \cite[Theorem 9.8]{McM03} that the conditions are also
satisfied for $D=21, 29, 33$. Before stating the main result
of this section, we need to introduce some notation.  As usual, we drop $D$ and
$\varepsilon$ from the notation.

Let $R$ be a finite extension of ${\mathcal O}_D$ over which
the family $f: \cXX \to C$ can be defined. For a set $\cSS$ of primes, 
we denote by $R_{\cSS}$ the set of $\cSS$-integers
in $R$.  There exists a finite set $\cSS$ of primes of
$R$ and a model $\ol{f}_\cSS:\ol{\cXX}_{R_\cSS}\to \ol{\cCC}_{R_\cSS}$
of $f$ such that for every prime $\wp$ which is invertible in $R_\cSS$
the reduction $\ol{f}_\cSS\otimes_{R_\cSS} R/\wp$ of
$\ol{f_\cSS}$ modulo $\wp$ is a family of stable curves of genus $2$
with the same number of degenerate fibers as $\ol{f}$. More precisely,
we require that the set, $S_i$, of singularities of $L_i$ extends to
an \'etale divisor over $\Spec(R_\cSS)$. (This may be accomplished by
extending the set $\cSS$, if necessary.)

In \S~\ref{reductionsec}, we explicitly determine such a set $\cSS$ in the
case that $D=13$ or $D=17$. We remark that the set $\cSS$ depends on
the choice of a model for $\ol{\cXX}$. 

  It follows from Lemma \ref{Fuchslem} that $L_i$ has a unique
holomorphic solution $u_i\in K[[t]]$ with $u_i(t=0)=1$. 
Here
one uses that $t=0$ is a cusp of $L_i$, and that the local exponents
of $L_i$ at $t=0$ are $(0,0)$.  The goal of this section is to prove
the following theorem.

\begin{Thm}\label{integralthm}
The unique holomorphic solution $u_i$ of $L_i$ with $u_i(t=0)=1$ has
coefficients in $ R_\cSS[[t]]$.
\end{Thm}

As a first step in the proof we determine the recursion relation
satisfied by the coefficients of $u_i$.

 Write $u_i=\sum_{j\geq 0} u_j^{(i)} t^j$. We let $S_i$ be the set
of singularities of $L_i$, and put $r_i=|S_i|$. We write $(\gamma_i,
\gamma_i)$ for the local exponents of $L_i$ at $t=\infty$. Let
$M_i=\prod_{\tau\neq 0, \infty} \tau$ be the product of the singularities of
$L_i$ different from $0,\infty$.  In the case that the differential operator
$L_i$ is given explicitly, these invariants may be read off from the
explicit expression for $L_i$.

Recall that $r_1$ is the number of cusps of $\ol{C}$, and
$r_2=r_1+\mu$, where $\mu$ is the number of zeros of the
Kodaira--Spencer morphism of $\cEE_2$. In the case that $D=13, 17$, we
have that $r_1=5$ and $r_2=7$. The following lemma expresses the local
exponent $\gamma_i$ at $\infty$ in terms of known invariants. Recall
that $\lambda_2$ is the Lyapunov exponent which was introduced in
\S~\ref{lyasec}.

\begin{Lemma}\label{locexplem}
\begin{itemize}
\item[(a)] We have $\gamma_1=(r_1-2)/2$.
\item[(b)] Let $\mu_j$ be a zero of the Kodaira--Spencer map $\Theta$
of $L_2$ and let $\delta_j$ the order of vanishing of $\Theta$ at
$\mu_j$. Then the local exponents of $L_2$ at $t=\mu$ are $(0,
\delta_j+1)$.
\item[(c)] The local exponent of $L_2$ at $t=\infty$ satisfies
\[
\gamma_2=\lambda_2(r_1-2)/2.
\]
\end{itemize}
\end{Lemma}

\begin{proof} Part (a) follows from the Riemann relation together with the 
assumption that the local exponents of $L_1$ at the cusps different from 
$\infty$ are $(0,0)$. 
Part (b) is proved in \cite{BoMo05}, Proposition 2.2.(b). 

Let $\{\mu_j\}$ be the set of zeros of $\Theta$, and let $\delta_j$ be
the order of vanishing of $\Theta$ at $\mu_j$. Then 
\[
\sum_j \delta_j=2g(\ol{C})-2+|S_1|-\lambda_2(r_1-2)=(r_1-2)(1-\lambda_2).
\]
Therefore the Riemann relation, together with (b) and our assumption
on the local exponents, implies that
\[
2\gamma_2=-\sum_j
(\delta_j+1)+|S_2|-2=r_1+\mu-2\sum_j\delta_j-\mu=\lambda_2(r_1-2).
\]
This proves the lemma.
\end{proof}

 In the case that $D=13, 17$, the Kodaira--Spencer morphism of $L_2$
has two simple zeros.  Lemma \ref{locexplem} implies in this case
that $\gamma_1=3/2$ and $\gamma_2=1/2$.  Moreover,
$M_1=1\cdot \tau\cdot \tau^{-1}=1$ and $M_2=M_1\cdot
\mu_1\cdot\mu_2=1$, as follows from the formulas in \S~\ref{eqnsec}
and \S~\ref{desec}. It might be possible to show, using Theorem
\ref{Galoisthm}, that one may choose a parameter $t$ on $\ol{C}$ such
that this holds in general.

\begin{Lemma}\label{reclem}
The coefficients $u_j^{(i)}$ of $u_i$ satisfy a recursion
\begin{equation}\label{receq}
D_{j, -1} u^{(i)}_{j+1}+D_{j, 0} u_j^{(i)}+\cdots +D_{j, r_i-3}
u^{(i)}_{j-r_i+3}=0,
\end{equation}
where $D_{j, -1}=\pm (j+1)^2M_i$ and 
$D_{j, r_i-3}=(j+\gamma_i-r_i+3)^2$.
\end{Lemma}

\begin{proof}
It is easy to see that the coefficients of a solution of a Fuchsian
differential equation satisfy a recursion as in (\ref{receq}). The
order of the recursion depends on our normalization of the local
exponents which assures that the numerator of $A_i$ (resp.\ $B_i$) has
degree $r_i-1$ (resp.\ $r_i-2$) in $t$. The formula for $D_{j, -1}$
and $D_{j, r_i-2}$ follows from an easy calculation, using the
assumption that the local exponents at $t=0$ are $(0,0)$.
\end{proof}

%
%

In the rest of this section, we fix a prime $p\not\in \cSS$ and let
$\wp|p$ be a prime of $R_\cSS$. We let $k=1$ if $D$ is a quadratic
residue $\pmod{p}$, and $k=2$ otherwise, i.e.\
$R_\cSS/\wp=\FF_{p^k}$. We write $\ZZ_{p^k}=W(\FF_{p^k})$ for the
$\wp$-adic completion of $R_\cSS$.  Let $\cRR$ be the $\wp$-adic
completion of $R_\cSS[t, 1/\tau_j, \quad \tau_j\in S_i)]$.  As in
\S~\ref{eqnsec}, we write $\ol{\cXX}_{\cRR}$ for the model of
$\ol{\cXX}$ over $\Spec(\cRR)$ defined by (\ref{eqneq}). We write
$\ol{\cXX}_\wp:=\ol{\cXX}_{\cRR}\otimes_\cRR (\cRR/\wp).$

As in \S~\ref{eqnsec}, we let
\[
y^2=g_t(x)=\sum_{k=0}^5 c_k x^k
\]
be an equation for $\ol{\cXX}$. The results of \S~\ref{lyasec} imply
that $\deg_t(c_k)=(r_1-2)(5-k)/3$. The definition of $\ol{C}$ implies
that these are integers.

\begin{Defi} Let $n$ be a natural number. We define polynomials
$B_{n, 1}\in R_\cSS[t]$ (resp.\ $B_{n, 2}\in
R_\cSS[t]$) as the coefficient of $x^{p^n-1}$ (resp.\
$x^{2p^n-1}$) in $g^{(p^n-1)/2}$. Similarly, we define $C_{n, 1}\in
R_\cSS[t]$ (resp.\ $C_{n, 2}\in R_\cSS[t]$) as the
coefficient of $x^{p^n-2}$ (resp.\ $x^{2p^n-2}$) in $g^{(p^n-1)/2}$.
\end{Defi}

\begin{Lemma}\label{expcoefflem}

\begin{itemize}
\item[(a)] The polynomials $B_{n,1}$ and $B_{n, 2}$ are solutions
$\pmod{p^n}$ of $L_1$.
\item[(b)] The polynomials $C_{n, 1}$ and $C_{n,2}$ are solutions
$\pmod{p^n}$ of $L_2$.
\end{itemize}
\end{Lemma}

\begin{proof}
 We note that $x$ is a local parameter of $\ol{\cXX}_\wp$ at $x=0$,
except at the $5$ zeros of $c_0$. We write
\[
\omega_1=\frac{{\rm d}x}{y}=g^{(p^n-1)/2}\frac{{\rm
d}x}{y^{p^n}}=\sum_{m\geq 0}P_m x^m \frac{{\rm d}x}{y^{p^n}x},
\]
and
\[
\omega_2=\frac{x\,{\rm d}x}{y}=g^{(p^n-1)/2}\frac{x\,{\rm
d}x}{y^{p^n}}=\sum_{m\geq 0}Q_m x^m \frac{{\rm d}x}{y^{p^n}x},
\]
where $P_m$ (resp.\ $Q_m$) is the coefficient of $x^{m-1} $ (resp.\
 $x^{m-2}$) in $g^{(p^n-1)/2}$. In particular $P_m, Q_m\in \cRR$.

The coefficients $P_m$ and $Q_m$ are called the {\em expansion
coefficients} of $\omega_i$. Katz (\cite{Katz84}) shows that $P_m$
(resp.\ $Q_m$) is a solution of $L_1$ (resp.\ $L_2$) modulo $m$,
i.e.\ of the same differential equation which is satisfied by
$\omega_1$ (resp.\ $\omega_2$). One may check this also directly, by
noting that $(\partial/\partial t)y^{p^n} \equiv 0\pmod{p^n}$.
Since $B_{n,k}=P_{kp^n}$ and $C_{n, k}=Q_{kp^n}$, it follows that these
 polynomials are solutions of $L_1$ and $L_2$ (modulo $p^n$), as well.
\end{proof}

\begin{Rem}\label{expcoefrem}
By using the formula for $\deg_t(c_k)$, one shows that the degree of
 $B_{n, 1}$ and $B_{n, 2}$ is less than or equal to $d_{n,
 1}:=(r_1-2)(p^n-1)/2$ and $d_{n,2}:=(r_1-2)(p^n-3)/6$, respectively.
 Similarly, the degree of $C_{n, 1}$ and $C_{n, 2}$ is less than or
 equal to $e_{n,1}:=(r_1-2)(3p^n-1)/6$ and
 $e_{n,2}:=(r_1-2)(p^n-1)/6$, respectively. In Remark~\ref{degrem} we
 show that equality holds.
\end{Rem}

\begin{Lemma}\label{RMlem}
\begin{itemize}
\item[(a)] Suppose that $D$ is a quadratic residue (mod $p$). Then
$$B_{1,2}\equiv C_{1,1}\equiv 0\pmod{p}.$$
\item[(b)] Suppose that $D$ is a quadratic nonresidue (mod $p$). Then
$$B_{1,1}\equiv C_{1,2}\equiv 0\pmod{p}.$$
\end{itemize}
\end{Lemma}

\begin{proof}
We consider $\omega_j$ as element of $\cEE_j\otimes_\cRR
(\cRR/\wp)\subset \cHH_\dR(\ol{\cXX}_\cRR)\otimes_{\cRR} (\cRR/\wp)$. We
denote by $\C$ the Cartier operator, and compute that
\[
\C \omega_1=\C \frac{x g^{(p-1)/2}}{y^p}\frac{{\rm
d}x}{x}=[B_{1,1}^{1/p}+B_{1,2}^{1/p}x]\frac{{\rm d}x}{y}, \quad
\C \omega_2=\C \frac{x^2 g^{(p-1)/2}}{y^p}\frac{{\rm
d}x}{x}=[C_{1,1}^{1/p}+C_{1,2}^{1/p}x]\frac{{\rm d}x}{y}.
\]
For the definition and  properties of the Cartier operator,
we refer to the article of Illusie in \cite{Hodge}.

 Suppose that $D$ is a quadratic residue (mod $p$). Then $\C$
stabilizes $\cLL_j\subset \cEE_j$. Therefore $B_{1,2}\equiv
C_{1,1}\equiv 0\pmod{p}$.

If $D$ is a quadratic nonresidue $\pmod{p}$, the Cartier operator
sends $\cLL_1\subset \cEE_1$ to $\cLL_2\subset \cEE_2$, and
conversely. This implies that $B_{1,1}\equiv C_{1,2}\equiv 0\pmod{p}$.
\end{proof}

\begin{Lemma}\label{congruencelem}
\begin{itemize}
\item[(a)] Suppose that $(D/p)=1$. Then
\[
B_{n+1, 1}\equiv B_{n, 1}^p B_{1,1} \pmod{p}, \qquad C_{n+1, 2}\equiv
C_{n, 2}^p C_{1,2} \pmod{p}.
\]
In particular, $B_{n, 1}\equiv B_{1,1}^{p^{n-1}+\cdots+p+1}\pmod{p}$
and $C_{n, 2}\equiv C_{1, 2}^{p^{n-1}+\cdots+p+1}\pmod{p}$.
\item[(b)] Suppose that $(D/p)=-1$. Then
\[
B_{n+1, 2}\equiv C_{n, 1}^p B_{1,2} \pmod{p}, \quad C_{n+1, 1}\equiv
B_{n, 2}^p C_{1,1} \pmod{p}.
\]
In particular, $B_{n, 2}\equiv C_{1,1}^{\cdots+p^3+p}\cdot
B_{1,2}^{\cdots+p^2+1}\pmod{p}$ and $C_{n, 1}\equiv
C_{1,1}^{\cdots+p^2+1}\cdot B_{1,2}^{\cdots+p^3+p}\pmod{p}$.
\end{itemize}
\end{Lemma}

\begin{proof}
We note that
\[
g^{(p^{n+1}-1)/2}=(g^p)^{(p^n-1)/2} \cdot g^{(p-1)/2}.
\]
The definition of the $B_{n, i}$ implies therefore that
\[
B_{n+1, 1}\equiv B_{n, 1}^p \cdot B_{1,1}+B_{n, 2}^p\cdot B_{1, 2}\pmod{p}.
\]
We used here that $\deg(g^{(p-1)/2})=5(p-1)/2$.
Assume now that $(D/p)=1$. Lemma \ref{RMlem} implies that
$B_{1,2}\equiv 0\pmod{p}$. Part (a) of the lemma follows
immediately. The other statements follow similarly.
\end{proof}

\begin{Lemma}\label{nonvanishinglem}
Let $k=1$ if $(D/p)=1$ and $k=2$ if $(D/p)=-1$.
 The polynomials $B_{n, k}$ and $C_{n, 2-k}$ are not identically zero
(modulo $p^n$).
\end{Lemma}

\begin{proof}
Lemma \ref{congruencelem} implies that it suffices to prove the lemma
for $n=1$. Let $t=t_i\neq \infty$ be a cusp. Denote by $\ol{\cXX}_{i,
\wp}$ the reduction of the degenerate fiber at $t=t_i$. Lemma
\ref{singfiberslem} implies that every irreducible component of
$\ol{\cXX}_{i, \wp}$ has geometric genus $0$.  It is well known that
this implies that $\ol{\cXX}_{i, \wp}$ is ordinary. (See for example
\cite{p-rank}, Lemma 1.3.)  Therefore the matrix of the Frobenius
morphism $F: H^1(\ol{\cXX}_\wp, {\mathcal O})\to H^1(\ol{\cXX}_\wp,
{\mathcal O})$ is invertible at $t=t_i$. Since the Cartier operator is
the transpose of the Frobenius under Serre duality, we conclude that
for $t_i\neq \infty$ we have that $B_{1,1}(t_i)\not\equiv 0\pmod{p}$.
\end{proof}

\begin{Rem}\label{degrem} The proof of Lemma \ref{nonvanishinglem}
 implies that the 
polynomials $B_{n, k}$ and $C_{n, 2-k}$ do not vanish at
$t=t_i\pmod{p^n}$ for all $i$ such that $t_i\neq \infty$. An analogous
argument, after replacing $t$ by a local parameter $1/t$ at
$t=\infty$, implies that the polynomial $B_{n, k}$ (resp.\ $C_{n,
2-k}$) has degree $d_{n, k}$ (resp.\ $e_{n, 2-k}$) $\pmod{p^n}$. Here
$d_{n, k}$ and $e_{2-k}$ are defined in Remark \ref{expcoefrem}.
\end{Rem}

We now apply Lemma \ref{reclem} to the operator $L_1$.  We write
\[
B_{n, k}=\sum_{i\geq 0} v_i^{(n)} t^i.
\]
Since $B_{n, k}$ is a solution of $L_1\pmod{p^n}$, the coefficients
$v_i^{(n)}$ satisfy the recursion (\ref{receq}) $\pmod{p^n}$.

\begin{Lemma}\label{betanlem}
Let $N=\lceil n/2\rceil$. Let $\beta_n$ be the smallest integer such
that $v_j^{(n)}\equiv 0\pmod{p^n}$ for $j=\beta_n+1, \beta_n+2,\ldots,
\beta_n+r_1-2$. Then $\beta_n\equiv -\gamma_1\equiv
-(r_1-2)/2\pmod{p^N}$. In particular,
\begin{equation}\label{betaneq}
\beta_n\geq \begin{cases} &\frac{p^N-r_1+2}{2}\quad \text{ if $r_1$ is odd},\\
&p^N-\frac{r_1-2}{2}\quad \text{ if $r_1$ is even}.
\end{cases}
\end{equation}
\end{Lemma}

\begin{proof}
The definition of $\beta_n$ implies that $v^{(n)}_{\beta_n}\not\equiv
0\pmod{p^n}$. Therefore the recursion (\ref{receq}) implies that
$D_{\beta_n+r_1-3, r_1-3}\equiv 0\pmod{p^n}$.  The statement of the lemma
follows now immediately from the formula for $D_{\beta_n+r_1-3, r_1-3}$ and
the fact that $\beta_n\geq 0$.
\end{proof}

\begin{Rem} \label{betanrem}
The proof of Lemma \ref{betanlem} implies that $\sum_{i=0}^{\beta_n}
v_j^{(n)} t^i$ is also a solution $\pmod{p^n}$, since its coefficients
satisfy the recursion (\ref{receq}) $\pmod{p^n}$. Since
$v^{(n)}_{\beta_n}\not\equiv 0\pmod{p^n}$, it follows that $L_1$ has a
solution $\pmod{p^n}$ of degree $\beta_n$. The inequality for
$\beta_n$ in Lemma \ref{betanlem} gives therefore a lower bound on the
degree of a polynomial solution $\pmod{p^n}$ of $L_1$. 
 
For the proof
of Theorem \ref{integralthm}, we do not need to know  this. We only need to
know that there exists a polynomial solution (modulo $p^n$) of degree
$\beta_n$, and that $\lim_{n\to \infty} \beta_n=\infty$.
\end{Rem}

\begin{Prop}\label{integralprop} 
For every $j\geq 0$, we have that $u_j^{(1)}\in \ZZ_{p^k}$.
\end{Prop}

\begin{proof}
We first fix an integer $n$.  Since $B_{n, k}(0)\not\equiv
0\pmod{p^n}$ (Lemma \ref{nonvanishinglem}), the polynomial $B_{n,
k}':=B_{n,k}/B_{n,k}(0)$ also has coefficients in $\ZZ_{p^k}$. The
coefficients of $B_{n, k}'$ still satisfy (\ref{receq}). We also  denote
these coefficients  by $v_j^{(n)}$. The definition of $\beta_n$
implies that for $0\leq j\leq \beta_n$ the coefficients $v_j^{(n)}\in
\ZZ_{p^k}$ are uniquely determined $\pmod{p^n}$ by $v_0^{(n)}=1$ and
$(\ref{receq})$. Since the $u_j^{(1)}$ satisfy the same recursion, we conclude
that
\[
u_j^{(1)}\equiv v_j^{(n)} \pmod{p^n}.
\]
In particular, we conclude that $u_j^{(1)}\in \ZZ_{p^k}$ for all
$j\leq \beta_n$.

For every $j\geq 0$, there exists an $n$ such that $j\leq \beta_n$
(Lemma \ref{betanlem}). Since $u_j^{(1)}\in K$, the proposition follows.
\end{proof}

Now we apply Lemma \ref{reclem} to the solution $u_2$ of $L_2$ which
has $r_2=r_1+\mu$ singularities.

Recall from \S
\ref{solsec} that $C_{n, 2-k}\in R_\cSS[t]$ is a solution $\pmod{p^n}$
of $L_2$ of degree $e_{n, 2-k}$.
 We write
\[
C_{n, k}=\sum_{j=0}^{e_{n, k}} w_j^{(n)} t^j.
\]
Note that the $w_j^{(n)}$ satisfy the recursion (\ref{receq}) $\pmod{p^n}$.

Let $\gamma_n$ be the smallest integer such that $w_j^{(n)}\equiv
0\pmod{p^n}$ for $j=\gamma_n+1, \ldots \gamma_n+r_2-2$. The recursion
(\ref{receq}) implies that $D_{\gamma_n+r_2+3, r_2-3}\equiv
0\pmod{p^n}$. This proves the following lemma which is an analog of
Lemma \ref{betanlem}, by using the expression for the local exponent of
$L_2$ at $\infty$ we gave in Lemma \ref{locexplem}.

\begin{Lemma}\label{gammanlem}
Let $N=\lceil n/2\rceil$. Then $\gamma_n\equiv -\gamma_2\pmod{p^N}$. In
particular,
\[
\gamma_n\geq \frac{p^N-(r_1-2)\lambda_2}{2}
\]
\end{Lemma}

Note that the estimate for $\gamma_n$ need not be an integer, so one
may improve it a bit, depending on the values of $r_1$ and
$\lambda_2$. However, we do not need this.  The proof of Proposition
\ref{integralprop} caries now over to $L_2$.

\begin{Prop}\label{integralprop2}
For every $j\geq 0$, we have that $u^{(2)}_j\in \ZZ_{p^k}$.
\end{Prop}

\bigskip\noindent
{\it Proof of Theorem \ref{integralthm}.\ }
The  theorem follows immediately from Propositions
\ref{integralprop} and \ref{integralprop2}.
\hspace*{\fill} $\Box$ \vspace{1ex} \noindent

\section{Reduction of the families $\cXX^\varepsilon(D)$ to characteristic 
$p>0$}
\label{reductionsec}
Let ${\mathcal O}_D\subset \QQ(\sqrt{D})$ be an order of discriminant
$D$ and $R={\mathcal O}_D[1/2]$.  Let $p\neq 2$ be a prime number, and
let $\wp$ be a prime of $R$ above $p$. We denote by $R_\wp$ the
completion of $R$ at $\wp$.  We say that $\ol{f}:\ol{\cXX}\to \ol{C}$
has {\em good reduction} at $\wp$ if there exists a model
$\ol{f}_R:\ol{\cYY}_R\to\ol{\cCC}$ of $\ol{f}:\ol{\cXX}\to \ol{C}$ over
$\Spec(R_\wp)$ such that $\ol{f}_{R}\otimes_{R} (R/\wp)$ is a family
of stable curves of genus $2$ with $|S|$ singular fibers. We say that
$\ol{f}$ has {\em potentially good reduction} if such a model exists
after replacing $R$ by the completion of a finite extension.

Let $D\in \{13, 17\}$ and $\varepsilon\in \{0, 1\}$. In this section
we consider the reduction to characteristic $p>0$ of the stable
families of curves $\ol{\cXX}^\varepsilon(D)$ defined by (\ref{eqneq},
\ref{c17eq}) and (\ref{X13eq}, \ref{c13eq}), respectively. This
section should be considered as a complement to the results in the
previous sections, though it is logically independent of it. 

Recall that it follows from the general theory that the family
$\ol{\cXX}^\varepsilon(D)$ has good reduction at all but a finite set
of primes. In this section, we determine this set for $D=13, 17$. 
This proof relies
on the explicit equation for the family of curves
$\ol{\cXX}^\varepsilon(D)$. One expects a similar result to hold much
more generally. It would be interesting to give a geometric proof of
the results of this section, relying on the properties of
Teichm\"uller curves. This would give a much deeper insight into the
arithmetic properties of Teichm\"uller curves.  If $D$ and $\ve$ are 
understood, we drop them from the notation.

Since the coefficients of $f_D$ are in $ R$, the formulas
(\ref{eqneq}) and (\ref{X13eq}) define a model $\ol{\cXX}_R$ of
$\ol{\cXX}$ over $\Spec(R)$. One may ask whether this model already
reduces to a family of stable curves of genus $2$ with $|S|$ singular
fibers. In this section, we consider this question in the case that
$D\in \{13, 17\}$.

\begin{Prop}\label{reductionprop}
\begin{itemize}
\item[(a)] Let $D=17$.  Let $p\neq 2, 17$ be a prime number, and let
$\wp|p$ be a prime of $R$ above $p$. Then $\ol{f}$ has good reduction
at $\wp$.
\item[(b)] Let $D=13$. Let $p\neq 2, 3, 13$ be a prime number, and let
$\wp|p$ be a prime of $R$ above $p$. Then $\ol{f}$ has good reduction
at $\wp$.
\end{itemize}
\end{Prop}

\begin{proof}
We first consider the case $D=17$.
Let $\wp$ be as in the statement of the lemma.  One computes that the
discriminant of $g_{17}(x)=\sum_k c_k x^k$ is equal to
\[
-\frac{17^{10}}{2^{12}}(4+\sqrt{17})^{19}\left(\frac{5}{2}+\frac{1}{2}
\sqrt{17}\right)
\left(\frac{5}{2}-\frac{1}{2}\sqrt{17}\right)^{18}(2t-31+7\sqrt{17})^3
(64t-31-7\sqrt{17})^3(t-1)^4t^5.
\]
Since $N(4+\sqrt{17})=1$ and
$N(({5}+\sqrt{17})/2)=N(({5}-\sqrt{17})/2)=2$,  it
follows that $\ol{\cXX}\otimes_R (R/\wp)$ is a curve of genus $2$ for
generic $t$. 

Choose
\begin{equation}\label{taueq}
 \tau=\frac{31-7\sqrt{17}}{2}, \qquad \tau^{-1}=\frac{31+7\sqrt{17}}{64}.
\end{equation}
Then $S=\{0,1,\infty, \tau, \tau^{-1}\}$ is the set of cusps of $C$.
We claim that the
points $0, 1, \infty, \tau, \tau^{-1}$ are pairwise noncongruent
$\pmod{\wp}$. Namely, one computes that for every pair $P_1, P_2$ of
points we have that $N(P_1-P_2)$ is a power of $2$. 
This proves the lemma for $D=17$.

Now let $D=13$ and let $\wp$ be a prime of $R=R(13)$ above $p\neq
2,3,13$. The discriminant of $g_{13}=\sum_k d_k x^k$ is equal to
\[
-\frac{3^{12}13^{10}}{2^{60}}\left(-\frac{3}{2}+\frac{1}{2}\sqrt{13}
\right)^{30}
\left(\frac{1}{2}+\frac{1}{2}\sqrt{13}\right)^6t^4(128t^2+71\sqrt{13}t+128)^4
(t-1)^4.
\]
Since $N((-3+\sqrt{13})/2)=-1$ and $N((1+\sqrt{13})/2)=-3$, it
follows that $\ol{\cXX}\otimes_R (R/\wp)$ is a curve of genus $2$ for
generic $t$. 

The set of cusps is $\{ 0, 1, \infty, \rho, \rho^{-1}\}$ with 
\begin{equation}\label{rhoeq}
\rho=-\frac{71}{256}\sqrt{13}+\frac{1}{256}\sqrt{-3}, \qquad
\rho^{-1}=-\frac{71}{256}\sqrt{13}-\frac{1}{256}\sqrt{-3}.
\end{equation}
The proposition for $D=13$ now follows as for $D=17$.
\end{proof}

 For $D=13, 17$, we let $\cSS_D\subset \ZZ$ be the set of primes such
that the the model defined by (\ref{eqneq}, \ref{c17eq}) and
(\ref{X13eq}, \ref{c13eq}) defines a family of stable curves with
$|S|=5$ singular fibers. Proposition~\ref{reductionprop} implies that
$\cSS_D=\{2,3,13\}$ if $D=13$ and $\cSS_D=\{2, 17\}$ if $D=17$. We
call $\cSS_D$ the set of {\em exceptional primes} of the model
$\ol{\cXX}_R$. Note that $\cSS_D$ depends on the choice of the model.

 In the rest of this section, we discuss some partial results
on the reduction of $\ol{f}:\ol{\cXX}\to \ol{C}$ at the exceptional
primes $p\neq 2$. A more detailed description might allow one to
extend the proof of the integrality to the exceptional primes $p\neq
2$.

\begin{Prop}\label{p=Dredprop}
Let $D\in \{13, 17\}$.  The family of curves
$\ol{f}^\varepsilon(D):\ol{\cXX}^\varepsilon(D)\to \ol{C}(D)$ has
potentially good reduction at $p=D$.
\end{Prop}

\begin{proof}
We only discuss the case $D=p=17$ and $\varepsilon=1$.  The argument in
the other cases is the same. Let
$\wp=(\sqrt{p})$ be the unique prime of $R$ above $p$. We remark that
$g_{D}^\varepsilon(x)\equiv (x+4t+4)^5\pmod{\wp}$.  Substituting
$x=z-4(t+1)$ yields $g_{D}^\varepsilon(z)\equiv z^5\pmod{\wp}$. By
considering the Newton polygon of $g_{D}^\varepsilon(z)$, we find that
$g_{D}^\varepsilon(z)$ has one root of valuation $v(p)$, and $4$ roots
of valuation $v(\sqrt{p})$. Therefore we substitute $z=\sqrt{p}w$ and
compute
\[
\tilde{g}_{D}^\varepsilon(w):=g_D^ve(w)/p^{3/2}\equiv
3(t^2+3t+1)(t^2+7t+1)w+5(t^2+3t+1)w^3+w^5.
\]

Performing the corresponding coordinate substitution for $y$ as well,
one obtains the equation
\[
v^2=\tilde{g}_{D}^\varepsilon(w)=3(t^2+3t+1)(t^2+7t+1)w+5(t^2+3t+1)w^3+w^5
\]
which defines a model of $\ol{\cXX}^\varepsilon(D)$ over
$R_{p}[p^{1/4}]$ whose fiber at $\wp$ is smooth. This shows that
$\ol{\cXX}^\varepsilon(D)$ has potentially good reduction at $\wp$.
\end{proof}

\begin{Lemma}\label{p=3redlem}
The family $\ol{f}:\ol{\cXX}(13)\to \ol{C}(13)$ does not have
potentially good reduction at $p=3$.
\end{Lemma}

\begin{proof}
Let $D=13$ and $p=3$.  We compute the stable model of the generic
fiber of $\ol{f}(13)$.  Let $\cRR$ be the completion of
$\ZZ[\sqrt{13}](t)$ at a prime above $p=3$. We let $X_3$ be the fiber of
$\ol{\cXX}(13)$ above the generic point of $\ol{C}(13)$. We claim
that, after replacing $\cRR$ by the completion of a finite extension,
there exists a stable model of $X_3$ over $\Spec(\cRR)$ whose special fiber
consists of two elliptic curves intersecting in one point. This
follows by explicitly blowing up the equation of the curve $X_3$, as in
the proof of Proposition~\ref{p=Dredprop}. The lemma follows from this 
and the uniqueness of the stable model.
\end{proof}

It is interesting that in the case that $D=13$ there is another prime
besides $D$ and $2$ which is exceptional for any choice of the model
of $\cXX_D$.

\begin{Question}
Is there a model of $\cXX_D$ such that $\cSS$ is minimal and if yes, what
is this set $\cSS$? Is there a canonical model of $\cXX_D$, comparable to 
the one for Shimura curves?
\end{Question}


\bigskip\noindent
Irene I.~Bouw \hfill  Martin M{\"o}ller \newline 
Institut f\"ur reine Mathematik \hfill Max-Planck-Institut f\"ur Mathematik
\newline
Helmholtzstra\ss e 18 \hfill Vivatsgasse 7 \newline
89069 Ulm \hfill 53111 Bonn \newline
irene.bouw@uni-ulm.de \hfill moeller@mpim-bonn.mpg.de \newline

\end{document}